\documentclass[12pt]{article}
\usepackage[utf8]{inputenc}
\usepackage{amsthm}
\usepackage{amsmath}
\usepackage{mathtools}
\usepackage{bbm}
\usepackage{amsfonts}
\usepackage{amssymb}
\usepackage[left=2cm,right=2cm,top=2cm,bottom=2cm]{geometry}

\usepackage{hyperref}

\usepackage{tabto}
\TabPositions{2 cm, 4 cm, 6 cm, 8 cm}

\usepackage{tikz-cd}

\usepackage{tikz}
\usetikzlibrary{arrows, automata, positioning}

\usepackage{tocloft}
\usepackage{appendix}

\usepackage{enumerate}

\usepackage{hyperref}
 
\newtheorem{theorem}{Theorem}[section]
\newtheorem{corollary}[theorem]{Corollary}
\newtheorem{lemma}[theorem]{Lemma}

\newtheorem{proposition}[theorem]{Proposition}

\theoremstyle{definition}
\newtheorem{definition}[theorem]{Definition}
\newtheorem{example}[theorem]{Example}
\newtheorem{remark}[theorem]{Remark}

\newtheorem{question}[theorem]{Question}

\newtheorem{notation}[theorem]{Notation}

\newtheorem*{lemma*}{Lemma}
\newtheorem*{proposition*}{Proposition}
\newtheorem*{theorem*}{Theorem}
\newtheorem*{corollary*}{Corollary}

\newcommand{\im}{\operatorname{im}}
\newcommand{\End}{\operatorname{End}}

\newcommand{\Ab}{\operatorname{Ab}}

\newcommand{\id}{\operatorname{id}}

\newcommand{\BS}{\operatorname{BS}}

\newcommand{\Aut}{\operatorname{Aut}}

\newcommand{\ZZ}{\operatorname{Z}}

\begin{document}

\title{Hopfian wreath products \\ and the stable finiteness conjecture}
\author{Henry Bradford and Francesco Fournier-Facio}
\date{\today}
\maketitle

\begin{abstract}
    We study the Hopf property for wreath products of finitely generated groups, focusing on the case of an abelian base group. 
    Our main result establishes a strong connection between this problem and Kaplansky's stable finiteness conjecture. Namely, the latter holds true if and only if for every finitely generated abelian group $A$ and every finitely generated Hopfian group $\Gamma$ the wreath product $A \wr \Gamma$ is Hopfian. In fact, we characterize precisely when $A \wr \Gamma$ is Hopfian, in terms of the existence of one-sided units in certain matrix algebras over $\mathbb{F}_p[\Gamma]$, for every prime $p$ occurring as the order of some element in $A$.

    A tool in our arguments is the fact that fields of positive characteristic locally embed into matrix algebras over $\mathbb{F}_p$ thus reducing the stable finiteness conjecture to the case of $\mathbb{F}_p$.
    A further application of this result shows that the validity of Kaplansky's stable finiteness conjecture is equivalent to a version of Gottschalk's surjunctivity conjecture for additive cellular automata.
\end{abstract}

\section{Introduction}

Given two discrete groups $\Delta$ and $\Gamma$, their (standard, restricted) \emph{wreath product} is the group $\Delta \wr \Gamma \coloneqq \bigoplus_\Gamma \Delta \rtimes \Gamma$, where $\Gamma$ acts on the direct sum by shifting coordinates. Wreath products are part of the standard toolbox in modern combinatorial and geometric group theory, providing examples of interesting behaviour, while still being intuitive and susceptible to explicit computations.

\medskip

One of the first and most influential occurrences of wreath products in combinatorial group theory was in Gruenberg's study of residual properties of solvable groups \cite{grunberg}. Along the way, he showed the that a wreath product of two (not necessarily solvable) groups $\Delta \wr \Gamma$ is residually finite if and only if $\Delta$ and $\Gamma$ are residually finite, and either $\Delta$ is abelian or $\Gamma$ is finite. When $\Delta$ and $\Gamma$ are both finitely generated, $\Delta \wr \Gamma$ is also finitely generated, and so in particular if $\Delta$ is abelian and $\Gamma$ is residually finite, then $\Delta \wr \Gamma$ is Hopfian.

Recall that a group is \emph{Hopfian} if every self-epimorphism is an automorphism. They are named after Hopf, who first showed that surface groups are Hopfian (see \cite[p. 415]{CGT}). The fact that finitely generated residually finite groups are Hopfian was proven by Mal'cev \cite{malcev}. The first examples of finitely generated non-Hopfian groups were given by B.H. Neumann \cite{fg:nonhopfian}, and of finitely presented ones by Higman \cite{fp:nonhopfian}, but the simplest and most important example is probably that of Baumslag--Solitar groups, such as $\BS(2, 3) = \langle a, t \mid ta^2t^{-1} = a^3 \rangle$ \cite{BS, BS:full, BS:full:detail}. The class of Hopfian groups has received much attention in geometric group theory ever since, and the connection with residual finiteness is an important reason for this. A standout example is the fact that hyperbolic groups are Hopfian \cite{hyphopf1, hyphopf2}, which may be seen as a weaker version of the longstanding conjecture that all hyperbolic groups are residually finite. In another direction, let us mention that the Hopf property is undecidable \cite{undecidable}, and this is not a straightforward application of the Adian--Rabin Theorem, since the Hopf property is neither Markov nor co-Markov \cite{hopf:embedding}.

\medskip

The starting point of this paper is the following general question:

\begin{question}
\label{q:main}

Let $\Delta, \Gamma$ be finitely generated groups. When is the wreath product $\Delta \wr \Gamma$ Hopfian?
\end{question}

Despite the fact that finitely generated residually finite groups are Hopfian, the two properties are very different in flavour. Indeed, residual finiteness ensures the existence of many normal subgroups, while Hopficity prevents the existence of certain normal subgroups. The most striking illustration of this, is that an infinite simple group cannot be residually finite, but it is Hopfian. Therefore, without any assumption of residual finiteness, trying to establish Hopficity of a group requires a different approach. Hence, while Question \ref{q:main} is related to Gruenberg's result, the techniques needed to attack it are bound to be different.

Let us start by mentioning some easy reductions for Question \ref{q:main}:

\begin{proposition}[Lemmata \ref{lem:D:Hopfian} and \ref{lem:G:Hopfian}]
\label{intro:prop:G:Hopfian}

If $\Delta \wr \Gamma$ is Hopfian, then $\Delta$ is Hopfian. If moreover $\Delta$ is abelian, then $\Gamma$ is Hopfian.
\end{proposition}

Given this, and the relevance of abelian groups in Gruenberg's result, for most of the paper we focus on the case in which $\Delta$ is a finitely generated abelian group. To make this assumption clear, we denote it by $A$ instead.

\medskip

As mentioned above, by Gruenberg's result, if $\Gamma$ is finitely generated 
residually finite then $A \wr \Gamma$ Hopfian. 
In light of this, and of Proposition \ref{intro:prop:G:Hopfian} above, 
it is natural to conjecture that $\Gamma$ finitely generated \emph{Hopfian} 
implies $A \wr \Gamma$ Hopfian. 
Our main result shows that this conjecture is equivalent to one of the most longstanding open problems in group theory:

\begin{theorem}[Theorem \ref{thm:main:general}]
\label{intro:thm:main:general}

The following are equivalent:
\begin{enumerate}
    \item For every finitely generated abelian group $A$ and every finitely generated Hopfian group $\Gamma$, the wreath product $A \wr \Gamma$ is Hopfian.
    \item Kaplansky's direct finiteness conjecture holds.
\end{enumerate}
\end{theorem}

\begin{remark}
All of the statements in this paper assume $\Gamma$ to be finitely generated. Other less elegant assumptions on $\Gamma$ also ensure that our main results hold (see Remark \ref{rem:fg}). 
\end{remark}

Recall that a ring with identity $R$ is \emph{directly finite} if every element with a one-sided inverse is a unit; equivalently, if $xy = 1$ implies $yx = 1$ (Lemma \ref{lem:left:right}). It is \emph{stably finite} if the matrix rings $\mathbb{M}_d(R)$ are directly finite for all $d \geq 1$. Kaplansky's direct (resp. stable) finiteness conjecture asserts that the group ring $\mathbb{F}[\Gamma]$ is directly (resp. stably) finite for every group $\Gamma$ and every field $\mathbb{F}$. It turns out that the two conjectures are equivalent \cite{df:dykemajuschenko} and hold over fields of characteristic $0$ \cite[p. 122]{kaplansky}. The case of positive characteristic is still wide open. It is known to hold for sofic groups \cite{df:elekszabo}, and more generally for surjunctive groups \cite{df:phung}, therefore a counterexample to the stable finiteness conjecture would also imply the existence of a non-sofic group, and would disprove Gottshalk's surjunctivity conjecture \cite{cell}. Moreover, a \emph{torsion-free} counterexample to the direct finiteness conjecture would also produce the first counterexample to Kaplansky's idempotent and zero-divisor conjectures, as well as a new counterexample to Kaplansky's unit conjecture, which was recently disproven by Gardam \cite{unit} (see also \cite{unit2}).

Because of this, Theorem \ref{intro:thm:main:general} is all the more striking: The innocent-looking group-theoretic Question \ref{q:main} shows a surprising connection to group rings and turns out to be at least as hard as some of the main open problems in modern group theory.

\medskip

A connection between morphisms of wreath products and properties of units in group rings had already appeared in the literature. This can be traced back to the remark that surjunctive groups satisfy Kaplansky's stable finiteness conjecture over finite fields \cite{df:elekszabo}, to which we shall also return. For a more recent example, in \cite{genevoistessera} the authors study morphisms to wreath products via geometric methods, and apply this to give a description of the automorphism group of $A \wr \Gamma$, where $A$ is a finite cyclic group and $\Gamma$ is a one-ended finitely presented group, in terms of the units of the group ring $A[\Gamma]$. Their starting point is similar to ours, but their analysis is completely geometric. Moreover, we believe that our methods could help to extend their analysis to the case in which the base is an arbitrary finite abelian group.

\medskip

Theorem \ref{intro:thm:main:general} is a consequence of the following more precise statement:

\begin{theorem}[Proposition \ref{prop:decomposition}, Corollary \ref{cor:solutionfreeabelian}, Theorem \ref{AbBaseIffThm}]
\label{intro:thm:main}

Let $\Gamma$ be a finitely generated group. Let $A$ be a finitely generated abelian group, decomposed as $A_0 \oplus \bigoplus_p A_p$, where $A_0$ is free abelian, and $A_p$ is a $p$-group for each prime $p$. Decompose further $A_p$ as $(\mathbb{Z}/p\mathbb{Z})^{d^p_1} \oplus \cdots \oplus (\mathbb{Z}/p^m\mathbb{Z})^{d^p_m}$, where $d^p_i \in \mathbb{N}$. Then the following are equivalent.
\begin{enumerate}
    \item $A \wr \Gamma$ is Hopfian.
    \item $\Gamma$ is Hopfian, and for each prime $p$ the ring $\mathbb{M}_{\max_i(d^p_i)}(\mathbb{F}_p[\Gamma])$ is directly finite.
\end{enumerate}
\end{theorem}

Appealing to solutions of the stable finiteness conjecture in certain cases, we deduce:

\begin{corollary}[Corollary \ref{cor:solution}]
\label{intro:cor:solution}
    Let $\Gamma$ be a finitely generated group, and let $A$ be a finitely generated abelian group. Suppose that one of the following holds:
    \begin{enumerate}
        \item $A$ is torsion-free;
        \item $\Gamma$ is sofic;
        \item $\Gamma$ is bi-orderable;
        \item $A$ is cyclic and $\Gamma$ has the unique product property.
    \end{enumerate}
    Then $A \wr \Gamma$ is Hopfian if and only if $\Gamma$ is Hopfian.
\end{corollary}

Up to embedding arguments, Theorem \ref{intro:thm:main} shows how Hopficity of such wreath products is equivalent to Kaplansky's stable finiteness conjecture over $\mathbb{F}_p$. To move to all fields of characteristic $p$, we use the following result:

\begin{proposition}[Proposition \ref{prop:LE}]
\label{intro:prop:LE}

Let $\mathbb{F}$ be a field of characteristic $p > 0$. Then $\mathbb{F}$ is locally embeddable into finite fields of characteristic $p$. In particular, $\mathbb{F}$ is locally embeddable into matrix algebras over $\mathbb{F}_p$.
\end{proposition}

This means that for every finite subset $K \subset \mathbb{F}$ there exists an integer $d \geq 1$ and a map $f \colon \mathbb{F} \to \mathbb{M}_d(\mathbb{F}_p)$ whose restriction to $K$ behaves like a ring homomorphism (see Definition \ref{defi:LE}).
The conclusion of Proposition \ref{intro:prop:LE} is well-known to experts but,
having been unable to locate in the literature of the precise formulation
we need, we include a self-contained proof below.
We apply this to our setting to obtain:

\begin{corollary}[Corollary \ref{cor:kaplansky:Fp}]
\label{intro:cor:kaplansky:Fp}

Let $p$ be a prime. Kaplansky's stable finiteness conjecture over some field of characteristic $p$ implies Kaplansky's stable finiteness conjecture over all fields of characteristic $p$.
\end{corollary}

After reading a preliminary version of this paper, Giles Gardam pointed out to us that a reduction to finite fields can be achieved for all of the Kaplansky Conjectures via a Nullstellensatz argument: See \cite[Section 3.4.3]{malman} for a proof in the case of the zero-divisor conjecture.
That said, a peculiarity of the stable finiteness conjecture is that it allows to move between matrix algebras of different degrees. This makes it possible to reduce to $\mathbb{F}_p$, rather than to the class of all finite fields. We do not know if such a strong reduction holds for the zero-divisor and idempotent conjectures.

\medskip

Corollary \ref{intro:cor:kaplansky:Fp} gives a new, elementary proof of the following fact:

\begin{corollary}[Corollary \ref{cor:surjunctive}]
\label{intro:cor:surjunctive}

Surjunctive groups satisfy Kaplansky's stable finiteness conjecture.
\end{corollary}

Indeed, it was already noticed in \cite{df:elekszabo} that surjunctive groups satisfy Kaplansky's stable finiteness conjecture over finite fields, which follows easily from the definition. The statement for surjunctive groups was recently proven by Phung \cite{df:phung} via operator algebras (see \cite{df:phung:2} for an alternative proof via model theory). Via the Gromov--Weiss Theorem \cite{gromov, weiss}, we also obtain a new proof of the following classical result:

\begin{corollary}[Corollary \ref{cor:surjunctive}]
\label{intro:cor:sofic}

Sofic groups satisfy Kaplansky's stable finiteness conjecture.
\end{corollary}

This statement has now been proven several times: Via operator algebras \cite{df:elekszabo}, via linear cellular automata \cite{df:ceccherinicoornaert} and via metric approximations of groups \cite{linearsofic}; Corollary \ref{intro:cor:sofic} adds a commutative algebra approach to this list of proofs.

\medskip

A further application of Proposition \ref{intro:prop:LE} is to provide yet another equivalence with the stable finiteness conjecture:

\begin{theorem}[Theorem \ref{thm:A:surj}]
\label{intro:thm:A:surj}
Let $\Gamma$ be a group. Then Kaplansky's stable finiteness conjecture holds over $\Gamma$ if and only if $\Gamma$ is $A$-surjunctive.
\end{theorem}

Here a group is said to be \emph{$A$-surjunctive} if every injective additive cellular automaton over $\Gamma$ is surjective (Definition \ref{defi:surjuctive}). The analogous statement for linear cellular automata, which yields the notion of \emph{$L$-surjunctive} groups, was proven in \cite{df:ceccherinicoornaert}.

\medskip

We are also able to treat some non-abelian base groups. On the one hand, we treat nilpotent groups, where an induction argument on the nilpotency class allows to extend our results in the abelian case:

\begin{theorem}[Theorem \ref{thm:nilpotentbasis} and Corollary \ref{cor:nilpotentbasis}]
\label{intro:thm:nilpotentbasis}

Let $\Delta$ be a finitely generated nilpotent group, with upper central series 
$\{ Z_i \}_{i = 0}^c$. Let $\Gamma$ be a finitely generated group and suppose that $(Z_i / Z_{i-1}) \wr \Gamma$ is Hopfian for all $i = 1, \ldots, c$. Then $\Delta \wr \Gamma$ is Hopfian.

In particular, Kaplansky's stable finiteness conjecture holds if and only if $\Delta \wr \Gamma$ is Hopfian for every finitely generated nilpotent group $\Delta$ and every finitely generated Hopfian group $\Gamma$.
\end{theorem}

On the opposite end of the spectrum, we examine the case in which $\Delta$ has certain properties that are incompatible with properties of $\Gamma$. In this case, Hopficity is easier to show directly, and leads to many examples of the following phenomenon. 

\begin{theorem}[Theorem \ref{thm:nonhopfiantop}]
\label{intro:thm:nonhopfiantop}

There exist finitely generated groups $\Delta, \Gamma$ such that $\Gamma$ is non-Hopfian but $\Delta \wr \Gamma$ is Hopfian.
\end{theorem}

This result shows that the abelian hypothesis in the second part of Proposition \ref{intro:prop:G:Hopfian} cannot be entirely removed. 
The proof is closely reminiscent of that of Gruenberg's criterion for non-residual finiteness 
of wreath products from \cite{grunberg}. 
In spite of Theorem \ref{intro:thm:nonhopfiantop}, 
the motivating Question \ref{q:main} is still open in its generality. 
In particular, we end with the following special case:

\begin{question}\label{q:final}

Do there exist finitely generated Hopfian groups $\Delta, \Gamma$ such that $\Delta \wr \Gamma$ is non-Hopfian?
\end{question}

In light of one direction of Theorem \ref{intro:thm:main:general}, 
such an example could be seen as a ``non-commutative'' negative answer 
to Kaplansky's stable finiteness conjecture. In view of our main results, 
to construct such an example one would expect that $\Delta$ should be far from abelian. So in particular we isolate the following special case of Question \ref{q:final}:

\begin{question}
Let $F$ be a non-abelian free group of finite rank, and let $\Gamma$ be a finitely generated Hopfian group. Is $F \wr \Gamma$ Hopfian?
\end{question}

This is a promising case to treat since, on top of being far from abelian, free groups are also far from satisfying any of the properties that are exploited in the proof of Theorem \ref{intro:thm:nonhopfiantop} (see also Remark \ref{rem:justnonP}).

\medskip

\textbf{Outline.} We start with some group-theoretic results on wreath products and their epimorphisms in Section \ref{s:wreath}. Then we move on to stable finiteness in Section \ref{s:stable}, proving Proposition \ref{intro:prop:LE} and its consequences. We combine the two approaches in Section \ref{s:main}, proving Theorem \ref{intro:thm:main}. Finally, we analyse non-abelian bases in Section \ref{s:nonabelian}, proving Theorems \ref{intro:thm:nilpotentbasis} and \ref{intro:thm:nonhopfiantop}.

\medskip

\textbf{Acknowledgements.} The authors are indebted to Giles Gardam, Anthony Genevois, Peter Kropholler, Markus Steenbock and John Wilson for useful conversations.
They also wish to thank the organisers of the conference \emph{YGGT X - Newcastle (Online)}, where this work was started.

\medskip

\textbf{Notations.} Throughout the rest of this paper, $\Delta$ and $\Gamma$ will always denote discrete groups, and $A$ will always denote an abelian group. We will use $\cdot$ to denote the multiplication in $\Gamma$, and $1_{\Gamma}$ to denote the identity, while for $A$ we use the additive notation $+, 0$. Conjugacy is denoted by ${_\gamma}f \coloneqq \gamma f \gamma^{-1}$, and accordingly commutators are defined as $[\gamma, f] \coloneqq {_\gamma}f \cdot f^{-1} = \gamma f \gamma^{-1} f^{-1}$. The set of natural numbers $\mathbb{N}$ contains $0$.

\section{Wreath products}
\label{s:wreath}

In this section we establish notations and terminology, and prove some group-theoretic facts about wreath products and their self-epimorphisms, before moving to the ring-theoretic approach in Section \ref{s:stable}.

\begin{definition}
The subgroup $\bigoplus_{\Gamma} \Delta \leq \Delta \wr \Gamma$ is called the \emph{base group}, and is denoted $\Delta[\Gamma]$. A subgroup of the base group is called \emph{basic}. A morphism $\varphi \colon \Delta \wr \Gamma \to \Delta \wr \Gamma$ is called \emph{basic} if the image of the base group is basic.
\end{definition}

We will denote elements of the base group as functions $f \colon \Gamma \to \Delta$ of finite support, assigning to each $\gamma \in \Gamma$ the corresponding coordinate. Then the action of $\Gamma$ takes the form ${_\gamma}f(x) = f(\gamma^{-1} x)$, and the group operation is:
$$(f_1, \gamma_1) (f_2, \gamma_2) = (f_1 \cdot {_{\gamma_1}}f_2, \gamma_1 \gamma_2).$$
When specializing to abelian bases, we will use $+$ to denote the operation on the base group, which makes the notation less confusing.

\begin{remark}
    Later on, we will make a connection between the base group $\Delta[\Gamma]$, when $\Delta$ is isomorphic to the additive group of a field $\mathbb{F}$, and the group ring $\mathbb{F}[\Gamma]$. To make the distinction clear, we will denote elements of the group ring not as functions but as formal linear combinations of group elements.
\end{remark}

\begin{notation}
Given $\delta \in \Delta$ and $\gamma \in \Gamma$ we denote by $\delta_{(\gamma)} \in \Delta[\Gamma]$ the element defined by
$$\delta_{(\gamma)}(x) = 
\begin{cases}
\delta \text{ if } x = \gamma; \\
1_{\Delta} \text{ otherwise}.
\end{cases}
$$
We denote by $\Delta_{(\gamma)} \coloneqq \{ \delta_{(\gamma)} \colon \delta \in \Delta \}$, that is the copy of $\Delta$ sitting at the coordinate $\gamma$.
\end{notation}

Notice that $\Delta \wr \Gamma$ is generated by $\Gamma$ and $\Delta_{(\gamma)}$ for any choice of $\gamma \in \Gamma$. In particular, if $\Delta$ and $\Gamma$ are finitely generated, then so is $\Delta \wr \Gamma$. One can also show that the converse is true, however this is less relevant for our purposes.

\subsection{Constructions of morphisms}

Here we prove Proposition \ref{intro:prop:G:Hopfian}; the constructions will also be used later on. First, we show that if $\Delta$ is non-Hopfian, then $\Delta \wr \Gamma$ is non-Hopfian:

\begin{lemma}
\label{lem:D:Hopfian}

Let $\varphi \colon \Delta \to \Delta'$ be an epimorphism. Then
$$\Phi \colon \Delta \wr \Gamma \to \Delta' \wr \Gamma : (f, \gamma) \mapsto (\varphi \circ f, \gamma)$$
is an epimorphism. It is an isomorphism if and only if $\varphi$ is.

In particular, if $\Delta$ is non-Hopfian, then $\Delta \wr \Gamma$ is non-Hopfian.
\end{lemma}

The proof is elementary and left to the reader. The other case is more interesting, and was already noticed by Gruenberg \cite[Lemma 3.2]{grunberg}:

\begin{lemma}
\label{lem:G:Hopfian}

Suppose that $A$ is abelian. Let $\varphi \colon \Gamma \to \Gamma'$ be an epimorphism. Given $f \in A[\Gamma]$ define $\varphi^*(f) \in A[\Gamma']$ by:
$$\varphi^*(f) \colon \Gamma' \to A : x \mapsto \sum\limits_{y \in \varphi^{-1}(x)} f(y).$$
Then
$$\Phi \colon A \wr \Gamma \to A \wr \Gamma' : (f, \gamma) \mapsto (\varphi^*(f), \varphi(\gamma))$$
is an epimorphism. It is an isomorphism if and only if $\varphi$ is.

In particular, if $\Gamma$ is non-Hopfian, then $A \wr \Gamma$ is non-Hopfian.
\end{lemma}

\begin{proof}
We start by noticing that $\varphi^* \colon A[\Gamma] \to A[\Gamma']$ is a homomorphism, and moreover
$$\varphi^*({_\gamma}f)(x) = \sum\limits_{y \in \varphi^{-1}(x)} f(\gamma^{-1}y) = \sum\limits_{z \in \varphi^{-1}(\varphi(\gamma^{-1}) x)} f(z) = \varphi^*(f)(\varphi(\gamma^{-1}) x);$$
therefore $\varphi^*({_\gamma}f) = {_{\varphi(\gamma)}}\varphi^*(f)$. This easily imply that $\Phi$ is a homomorphism:
\begin{align*}
\Phi((f_1, \gamma_1)(f_2, \gamma_2)) = \Phi(f_1 + {_{\gamma_1}}f_2, \gamma_1 \gamma_2) &= (\varphi^* (f_1) + {_{\varphi(\gamma_1)}}\varphi^*(f_2), \varphi(\gamma_1) \varphi(\gamma_2))\\
&= (\varphi^*(f_1), \varphi(\gamma_1)) (\varphi^*(f_2), \varphi(\gamma_2)) = \Phi(f_1, \gamma_1) \Phi(f_2, \gamma_2).
\end{align*}

Since $\Phi(\Gamma) = \varphi(\Gamma)$, and $\Phi(\Delta_{(\gamma)}) = \Delta_{(\varphi(\gamma))}$, the surjectivity of $\varphi$ implies the surjectivity of $\Phi$. Since $\Phi|_{\Gamma} = \varphi$, it follows that $\Phi$ can only be injective if $\varphi$ is injective. Conversely, if $\varphi$ is injective, then it is an isomorphism, and so we may rewrite $\Phi(f, \gamma) = (f \circ \varphi^{-1}, \varphi(\gamma))$, which shows that $\Phi$ is injective.
\end{proof}

We will see in Theorem \ref{thm:nonhopfiantop} that the assumption that the base group be abelian is necessary. Indeed, if $\Delta$ is non-abelian, it is possible for $\Delta \wr \Gamma$ to be Hopfian even if $\Gamma$ is itself not Hopfian. Here the assumption was used to make sense of the definition of $\varphi^*$ by means of a sum of finitely many elements. However, this ambiguity only occurs in the case in which $\varphi$ is not injective.

\subsection{Basic morphisms}

Recall that a self-epimorphism is called basic if it sends the base group to inside the base group. The ring-theoretic approach of the next sections will work only with basic epimorphisms, therefore the following result is an important starting point, which will crucially use the fact that the base group is abelian:

\begin{proposition}
\label{prop:morphismsarebasic}

Let $A$ be an abelian group, and let $\Gamma$ be a finitely generated infinite group. Then every epimorphism $A \wr \Gamma \to A \wr \Gamma$ is basic.
\end{proposition}

The hypothesis that $\Gamma$ is infinite is needed, as the next example shows:

\begin{example}
There is a non-basic automorphism $\varphi \in \Aut( \mathbb{Z}/2\mathbb{Z} \wr \mathbb{Z}/2\mathbb{Z} )$ given by: 
\[\varphi \colon ((x, y),0) \mapsto ((x, x), x+y) \text{ and } \varphi((0,0),1) = ((1,0),0).\]
Under the isomorphism $\mathbb{Z}/2\mathbb{Z} \wr \mathbb{Z}/2\mathbb{Z} \cong D_8$, where $D_8$ is the dihedral group of order $8$, this is the outer automorphism given by swapping vertices and edges of the square on which $D_8$ acts.
\end{example}

We start with the following lemma.

\begin{lemma}
\label{lem:abeliannormal}

Let $A$ be a non-trivial abelian group. Let $N \leq A \wr \Gamma$ be an abelian normal subgroup that is not basic. Then:
\begin{enumerate}
    \item $A$ has exponent $2$;
    \item There exists a central element $\gamma \in \Gamma$ of order $2$;
    \item $N$ equals the kernel of the epimorphism $A \wr \Gamma \to A \wr (\Gamma/\langle \gamma \rangle)$ from Lemma \ref{lem:G:Hopfian}.
\end{enumerate}
\end{lemma}

\begin{proof}
Suppose that $(f, \gamma) \in N$ and $\gamma \neq 1_{\Gamma}$. For every non-identity $a \in A$ and every $g \in \Gamma$, using that $A$ is abelian and $N$ is normal:
$$[(f, \gamma), a_{(g)}] = a_{(\gamma g)} - a_{(g)} \in N.$$
Moreover, using that $N$ is abelian:
$$a_{(\gamma g)} - a_{(g)} = (f, \gamma)(a_{(\gamma g)} - a_{(g)})(f, \gamma)^{-1} = a_{(\gamma^2 g)} - a_{(\gamma g)}.$$
Taking $g = 1_{\Gamma}$, since $\gamma \neq 1_{\Gamma}$ we obtain $\gamma^2 = 1_{\Gamma}$ and $a = -a$. This shows that $A$ has exponent $2$ and that $\gamma$ has order $2$.

Now suppose that $(f_1, \gamma_1) \in N$. Then we have
$$a_{(\gamma)} - a_{(1_{\Gamma})} = (f_1, \gamma_1)(a_{(\gamma)} - a_{(1_{\Gamma})})(f_1, \gamma_1)^{-1} = a_{(\gamma_1 \gamma)} - a_{(\gamma_1)}.$$
Therefore either $\gamma_1 = \gamma$ or $\gamma_1 = 1_{\Gamma}$. This shows that the $\Gamma$-coordinate of an element of $N$ belongs to $\langle \gamma \rangle$. In particular, since $N$ is normal, we deduce that $\gamma$ is central.

Finally, consider the epimorphism $\pi \colon A \wr \Gamma \to A \wr (\Gamma / \langle \gamma \rangle)$, as defined in Lemma \ref{lem:G:Hopfian}. Explicitly, this is defined on $\Gamma$ as the quotient $\pi \colon \Gamma \to \Gamma / \langle \gamma \rangle$ and on $A[\Gamma]$ as the map $\pi^* \colon A[\Gamma] \to A[\Gamma / \langle \gamma \rangle]$ where $\pi^*(f)(\pi(x)) = f(x) + f(x \gamma)$. The kernel $K$ consists of elements $(f_1, \gamma_1)$, where $\gamma_1 \in \langle \gamma \rangle$ and $f_1$ is such that $f_1(x) + f_1(\gamma x) = 0$ for all $x \in \Gamma$. Since $A$ has exponent $2$, this is equivalent to $f_1$ being $\gamma$-invariant. Now if $(f_1, \gamma_1) \in N$, then we have already seen that $\gamma_1 \in \langle \gamma \rangle$, and conjugating by $(f, \gamma) \in N$ shows that $f_1$ is $\gamma$-invariant. Therefore $N \leq K$.
For the other inclusion, a finitely supported $\gamma$-invariant function $f$ may be written as a sum of elements of the form $a_{(\gamma g)} - a_{(g)}$. Since all of these belong to $N$, we have $f \in N$. In particular, taking $(f, \gamma) \in N$ we also have $\gamma \in N$. It follows that $K \leq N$, which concludes the proof.
\end{proof}

Next, we recall two well-known structural results about wreath products. The first concerns the centre:

\begin{lemma}
\label{lem:centre}

If $\Gamma$ is infinite and $\Delta$ is non-trivial, then $\Delta \wr \Gamma$ is centreless.
\end{lemma}

\begin{proof}
Suppose that $(f, \gamma) \in \ZZ(\Delta \wr \Gamma)$. Then for every $g \in \Gamma$ it holds $(f, \gamma) = {_g}(f, \gamma) = ({_g}f, {_g}\gamma)$. Identifying the first coordinates shows that $f$ is $\Gamma$-invariant, and since $\Gamma$ is infinite this implies that $f = 0$, therefore $\gamma \in \ZZ(\Gamma \wr \Delta)$. But then $\delta_{(g)} = {_\gamma}\delta_{(g)} = \delta_{(\gamma g)}$ and so $\gamma g = g$ for all $g \in \Gamma$, which implies $\gamma = 1_{\Gamma}$.
\end{proof}

The second concerns the abelianisation.

\begin{definition}
Given an abelian group $A$ and a group $\Gamma$, the \emph{augmentation map} is the homomorphism
$$\varepsilon \colon A[\Gamma] \to A : f \mapsto \sum\limits_{x \in \Gamma} f(x).$$
\end{definition}

\begin{lemma}
\label{lem:abelianisation}

If $A$ is abelian, the abelianisation of $A \wr \Gamma$ is given by:
$$A \wr \Gamma \to A \times \Ab(\Gamma) : (f, \gamma) \mapsto (\varepsilon(f), \Ab(\gamma)).$$
\end{lemma}

\begin{proof}
The kernel of the above map is the group $K \coloneqq \{ (f, \gamma) : \varepsilon(f) = 0, \gamma \in [\Gamma, \Gamma] \}$. Therefore it suffices to show that if $\varphi \colon A \wr \Gamma \to B$ is a homomorphism and $B$ is abelian, then $\varphi(K) = 0_B$. For the $\Gamma$-coordinate this is tautological, so it suffices to show that if $\varepsilon(f) = 0$, then $\varphi(f) = 0_B$. This follows from the following computation:
\begin{align*}
    \varphi(f) &= \varphi \left( \sum\limits_{x \in \Gamma} f(x)_{(x)} \right) = \varphi \left( \sum\limits_{x \in \Gamma} {_x}(f(x)_{(e)}) \right) = \sum\limits_{x \in \Gamma} {_{\varphi(x)}}\varphi(f(x)_{(e)}) \\
    &= \sum\limits_{x \in \Gamma} \varphi(f(x)_{(e)}) = \varphi \left( \sum\limits_{x \in \Gamma} f(x)_{(e)} \right) = \varphi (\varepsilon(f)_{(e)}).
\end{align*}
Note that all expressions above are well-defined, since $f(x) = 0_A$ for all but finitely many $x \in \Gamma$.
\end{proof}

We are finally ready to prove the main result of this subsection.

\begin{proof}[Proof of Proposition \ref{prop:morphismsarebasic}]
Let $\varphi \colon A \wr \Gamma \to A \wr \Gamma$ be an epimorphism, and suppose that it is not basic. If $A$ is trivial the statement is void, so let us assume that $A$ is non-trivial. Then $\varphi(A[\Gamma]) = N$ is abelian, normal and non-basic. Thus we are in the situation of Lemma \ref{lem:abeliannormal}, in particular $N$ equals the kernel of the epimorphism $A \wr \Gamma \to A \wr (\Gamma / \langle \gamma \rangle)$, where $\gamma \in \Gamma$ is central and of order $2$. Therefore $\varphi$ descends to an epimorphism
$$\overline{\varphi} \colon \Gamma \cong (A \wr \Gamma) / A[\Gamma] \to (A \wr \Gamma) / N \cong A \wr (\Gamma / \langle \gamma \rangle).$$
Since $\Gamma$ is infinite, so is $\Gamma / \langle \gamma \rangle$, thus $A \wr (\Gamma / \langle \gamma \rangle)$ is centreless by Lemma \ref{lem:centre}. But then $\gamma \in \ZZ(\Gamma)$ must be in the kernel of $\overline{\varphi}$, which therefore induces a further epimorphism
$$\Gamma / \langle \gamma \rangle \to A \wr \Gamma / \langle \gamma \rangle.$$

Passing to the abelianisation, by Lemma \ref{lem:abelianisation} we obtain an epimorphism
$$\Ab(\Gamma / \langle \gamma \rangle) \to A \times \Ab(\Gamma / \langle \gamma \rangle)$$
of finitely generated abelian groups. Composing this with the (non-injective) projection onto $\Ab(\Gamma / \langle \gamma \rangle)$, we have produced a self-epimorphism of $\Ab(\Gamma / \langle \gamma \rangle)$ that is not injective. This shows that $\Ab(\Gamma / \langle \gamma \rangle)$ is non-Hopfian. But $\Gamma$, and thus $\Ab(\Gamma / \langle \gamma \rangle)$, is finitely generated, and finitely generated abelian groups are residually finite thus Hopfian.
\end{proof}

\begin{remark}
\label{rem:fg}
The hypothesis that $\Gamma$ be finitely generated is used throughout this paper only in order to apply Proposition \ref{prop:morphismsarebasic}. However, it is apparent from the proof that it is enough to assume that $\Gamma$ has a Hopfian abelianisation.
Other assumptions ensure that every epimorphism is basic just by using Lemma \ref{lem:abeliannormal}, for instance one may assume that $\Gamma$ has no central involution, or that $A$ does not have exponent $2$.
We stick to finite generation for the sake of elegance of the statements.
\end{remark}

\subsection{Some reductions}

The main goal of this subsection is to prove the following reduction result:

\begin{proposition}
\label{prop:decomposition}

Let $A$ be a finitely generated abelian group, and write $A = A_0 \oplus \bigoplus_p A_p$, where $p$ is ranges over all primes, $A_0$ is free abelian and $A_p$ is a finite abelian $p$-group. Let $\Gamma$ be a finitely generated Hopfian group. Then the following are equivalent:
\begin{enumerate}
    \item $A \wr \Gamma$ is Hopfian;
    \item $A_0 \wr \Gamma$ and $A_p \wr \Gamma$ are Hopfian, for every $p$.
\end{enumerate}
\end{proposition}

In fact, we will see in Corollary \ref{cor:solutionfreeabelian} that $A_0 \wr \Gamma$ is always Hopfian.

\medskip

Let us start by proving a general fact about self-epimorphisms of semidirect products, which will combine well with Proposition \ref{prop:morphismsarebasic}:

\begin{lemma}
\label{lem:epi:semi}

Let $\varphi \colon \Lambda \rtimes \Gamma \to \Lambda \rtimes \Gamma$ be an epimorphism, and suppose that $\varphi(\Lambda) \leq \Lambda$, and that $\Gamma$ is Hopfian. Then $\varphi(\Lambda) = \Lambda$ and $\ker(\varphi) \leq \Lambda$.

Suppose moreover that $\Lambda$ is abelian. Then there exists an automorphism $\alpha$ of $\Gamma$ such that
\[\psi \colon \Lambda \rtimes \Gamma \to \Lambda \rtimes \Gamma : (\lambda, \gamma) \mapsto (\varphi(\lambda), \alpha(\gamma))\]
is an epimorphism with $\ker(\psi) = \ker(\varphi)$. In particular, $\varphi$ is injective if and only if $\psi$ is injective.
\end{lemma}

\begin{proof}
Consider the following commutative diagram:
\[\begin{tikzcd}
	{\Lambda \rtimes \Gamma} && {\Lambda \rtimes \Gamma} \\
	\\
	\Gamma && {(\Lambda \rtimes \Gamma)/\varphi(\Lambda)}
	\arrow["\varphi", two heads, from=1-1, to=1-3]
	\arrow[two heads, from=1-1, to=3-1]
	\arrow[two heads, from=1-3, to=3-3]
	\arrow[two heads, from=3-1, to=3-3]
\end{tikzcd}\]
The right arrow is the quotient by the normal subgroup $\varphi(\Lambda)$. Precomposing with $\varphi$ gives a morphism whose kernel contains $\Lambda$, so this morphism factors through the quotient by $\Lambda$; this defines the lower arrow and completes the square.

Since $\varphi(\Lambda) \leq \Lambda$, the group $(\Lambda \rtimes \Gamma)/\varphi(\Lambda)$ admits $\Gamma$ as a quotient. If $\varphi(\Lambda)$ were a proper subgroup of $\Lambda$, then the lower arrow composed with this quotient would contradict that $\Gamma$ is Hopfian. Therefore $\varphi(\Lambda) = \Lambda$ and the commutative diagram above can be rewritten as follows:
\[\begin{tikzcd}
	{\Lambda \rtimes \Gamma} && {\Lambda \rtimes \Gamma} \\
	\\
	\Gamma && {\Gamma}
	\arrow["\varphi", two heads, from=1-1, to=1-3]
	\arrow[two heads, from=1-1, to=3-1]
	\arrow[two heads, from=1-3, to=3-3]
	\arrow["\cong", from=3-1, to=3-3]
\end{tikzcd}\]
Now if $(f, \gamma) \in \ker(\varphi)$, then $(f, \gamma)$ is in the kernel of the composition of the left arrow and the bottom arrow; since the latter is injective, this implies that $\gamma = 1_\Gamma$. In other words, $\ker(\varphi) \leq \Lambda$.

\medskip

For every endomorphism $\varphi$ of $\Lambda \rtimes \Gamma$, there exist maps $b \colon \Gamma \to \Lambda$ and $\alpha \colon \Gamma \to \Gamma$ such that $\varphi(\gamma) = (b(\gamma), \alpha(\gamma))$, and $\alpha$ is a homomorphism. The previous paragraph shows that, under our assumptions, $\alpha$ is a homomorphism. Suppose that $\Lambda$ is abelian. We define $\psi$ as in the statement, namely $\psi(\lambda, \gamma) \coloneqq (\varphi(\lambda), \alpha(\gamma))$ (recall that $\varphi(\Lambda) = \Lambda$) and check that it is a homomorphism. Clearly $\psi|_\Lambda = \varphi|_\Lambda$ and $\psi|_\Gamma = \alpha$ are homomorphisms. We check that $\psi$ satisfies the conjugacy relation:
$$\psi({_\gamma}\lambda) = \varphi({_\gamma}\lambda) = {_{\varphi(\gamma)}}\varphi(\lambda) = {_{(b(\gamma), \alpha(\gamma))}}\varphi(\lambda) = {_{\alpha(\gamma)}}\varphi(\lambda) = {_{\psi(\gamma)}}\psi(\lambda);$$
where we used that $\Lambda$ is abelian.

Since $\psi(\Gamma) = \alpha(\Gamma) = \Gamma$ and $\psi(\Lambda) = \varphi(\Lambda) = \Lambda$, it follows that $\psi$ is an epimorphism. Moreover, $\ker(\varphi), \ker(\psi) \leq \Lambda$ by the first statement, and $\psi|_\Lambda = \varphi|_\Lambda$, therefore $\ker(\psi) = \ker(\varphi)$.
\end{proof}

This takes the following form for wreath products, which will be useful later on:

\begin{proposition}
\label{prop:basic:morphism}

Let $A$ be an abelian group, let $\Gamma$ be an infinite finitely generated group, and let $\varphi \colon A \wr \Gamma \to A \wr \Gamma$ be an epimorphism. Then $\varphi(A[\Gamma]) = A[\Gamma]$ and $\ker(\varphi) \leq A[\Gamma]$.

Moreover, $\psi \colon (f, \gamma) \mapsto (\varphi(f), \gamma)$ is an epimorphism such that $\ker(\psi) = \ker(\varphi)$. In particular, $\varphi$ is injective if and only if $\psi$ is injective.
\end{proposition}

\begin{proof}
We write $A \wr \Gamma \cong A[\Gamma] \rtimes \Gamma$. Proposition \ref{prop:morphismsarebasic} shows that $\varphi(A[\Gamma]) \leq A[\Gamma]$, thus we may apply Lemma \ref{lem:epi:semi}, which shows the first part of the statement, and gives an automorphism $\alpha$ of $\Gamma$ such that $(f, \gamma) \mapsto (\varphi(f), \alpha(\gamma))$ is a self-epimorphism with the same kernel as $\varphi$. We can then postcompose this with the automorphism $(f, \gamma) \mapsto ((\alpha^{-1})^* (f), \alpha^{-1} (\gamma))$, as described in Lemma \ref{lem:G:Hopfian}, to conclude.
\end{proof}

Now let us move to the matter at hand, namely Proposition \ref{prop:decomposition}. In order to apply Lemma \ref{lem:epi:semi}, we start by noticing a decomposition for wreath products with base a direct product. Let $A, B$ be abelian groups, and let $\Gamma$ be a group. Then $(A \times B) \wr \Gamma \cong A[\Gamma] \rtimes (B \wr \Gamma)$, where $B \wr \Gamma$ acts on $A[\Gamma]$ by letting $B$ act trivially and $\Gamma$ act as usual. We will use the notation $(f_A, f_B, \gamma)$ to denote elements in this form.

We start with the easier direction:

\begin{lemma}
\label{lem:AB:Hopfian}

Let $A, B$ be abelian groups, and let $\Gamma$ be an infinite finitely generated group. If $(A \times B) \wr \Gamma$ is Hopfian, then $A \wr \Gamma$ and $B \wr \Gamma$ are Hopfian.
\end{lemma}

\begin{proof}
If $\Gamma$ is non-Hopfian then none of the groups above is Hopfian by Lemma \ref{lem:G:Hopfian}, so we may assume that $\Gamma$ is Hopfian. Suppose that $B \wr \Gamma$ is non-Hopfian: We will show that $(A \times B) \wr \Gamma$ is non-Hopfian (the other case follows by symmetry). Let $\varphi \colon B \wr \Gamma \to B \wr \Gamma$ be a self-epimorphism that is not injective. Using Proposition \ref{prop:basic:morphism}, we may assume that $\varphi$ is of the form $\varphi(f, \gamma) = (\varphi(f), \gamma)$.

Write $(A \times B) \wr \Gamma \cong A[\Gamma] \rtimes (B \wr \Gamma)$, and define
\[\Phi \colon A[\Gamma] \rtimes (B \wr \Gamma) \to A[\Gamma] \rtimes (B \wr \Gamma) : (f_A, f_B, \gamma) \mapsto (f_A, \varphi(f_B), \gamma).\]
This is a homomorphism when restricted to both $A[\Gamma]$ and $B \wr \Gamma$, and it is easily seen to satisfy the conjugacy relation, therefore it is a homomorphism. It is an epimorphism since $\Phi(A[\Gamma]) = A[\Gamma]$ and $\Phi(B \wr \Gamma) = \varphi(B \wr \Gamma) = B \wr \Gamma$, and it is not injective since $\Phi|_{B \wr \Gamma} = \varphi$, and the latter is not injective.
\end{proof}

Now we move to a partial converse of Lemma \ref{lem:AB:Hopfian}, which needs some additional assumptions.

\begin{lemma}
\label{lem:AxB:Hopfian}

Let $A, B$ be abelian groups. Suppose that $A$ is torsion, and any two torsion elements of $A$ and $B$ have coprime orders. Let $\Gamma$ be an infinite finitely generated group. If $A \wr \Gamma$ and $B \wr \Gamma$ are Hopfian, then $(A \times B) \wr \Gamma$ is Hopfian. In particular, this holds if $A$ is finite, and $B$ is either free abelian, or finite and of order coprime to $A$.
\end{lemma}

\begin{proof}
Again we may assume that $\Gamma$ is Hopfian by appealing to Lemma \ref{lem:G:Hopfian}. Suppose that $(A \times B) \wr \Gamma$ is non-Hopfian, and $B \wr \Gamma$ is Hopfian: We will show that $A \wr \Gamma$ is non-Hopfian. Let $\varphi \colon (A \times B) \wr \Gamma \to (A \times B) \wr \Gamma$ be a self-epimorphism that is not injective. Using Proposition \ref{prop:basic:morphism}, we may assume that $\varphi(f, \gamma) = (\varphi(f), \gamma)$, and $\varphi|_{(A \times B)[\Gamma]}$ is a self-epimorphism of $(A \times B)[\Gamma]$.

The restiction of $\varphi$ to $A[\Gamma] \to (A \times B)[\Gamma]$ actually has image in $A[\Gamma]$: Indeed, every element of $A[\Gamma]$ is torsion and its order does not divide the order of any element in $B[\Gamma]$. Therefore writing $(A \times B) \wr \Gamma \cong A[\Gamma] \rtimes (B \wr \Gamma)$, we have $\varphi(A[\Gamma]) \leq A[\Gamma]$. Since $B \wr \Gamma$ is Hopfian and $A[\Gamma]$ is abelian, we are in the setting of Lemma \ref{lem:epi:semi}, and thus $\varphi(A[\Gamma]) = A[\Gamma]$ and $\ker(\varphi) \leq A[\Gamma]$. Moreover we may assume that there exists an automorphism $\alpha$ of $B \wr \Gamma$ such that $\varphi(f_A, f_B, \gamma) = (\varphi(f_A), \alpha(f_B, \gamma))$. Applying Proposition \ref{prop:basic:morphism} to $\alpha$, we obtain an automorphism $\beta$ of $\Gamma$ and reduce to the case in which $\varphi$ takes the form $\varphi(f_A, f_B, \gamma) = (\varphi(f_A), \varphi(f_B), \beta(\gamma))$.

Define $\Phi \colon A \wr \Gamma \to A \wr \Gamma : (f, \gamma) \mapsto (\varphi(f_A), \beta(\gamma))$. Then $\Phi$ is a homomorphism, being the restriction of $\varphi$ to $A[\Gamma] \rtimes \Gamma$, it is an epimorphism since $\Phi(A[\Gamma]) = A[\Gamma]$ and $\Phi(\Gamma) = \Gamma$, and it is not injective since $\Phi|_{A[\Gamma]} = \varphi|_{A[\Gamma]}$ and $\ker(\varphi)$ is non-trivial and contained in $A[\Gamma]$. Therefore $A \wr \Gamma$ is not Hopfian and we conclude.
\end{proof}

We deduce the main result of this subsection.

\begin{proof}[Proof of Proposition \ref{prop:decomposition}]

If $\Gamma$ is finite, then $A \wr \Gamma$ is virtually finitely generated abelian, thus residually finite. In particular, $A \wr \Gamma$ is Hopfian for every finitely generated abelian group $A$. Therefore we may assume that $\Gamma$ is infinite, and then the result follows from Lemmata \ref{lem:AB:Hopfian} and \ref{lem:AxB:Hopfian}.
\end{proof}

\section{Stable finiteness}
\label{s:stable}

In this section we discuss direct and stable finiteness of group rings. We will start by recalling some known cases of Kaplansky's conjectures, and then prove Proposition \ref{intro:prop:LE} and its consequences: Corollaries \ref{intro:cor:kaplansky:Fp} and \ref{intro:cor:surjunctive}, and Theorem \ref{intro:thm:A:surj}.

\subsection{Known results}

Let us start by pointing out the following elementary but useful fact:

\begin{lemma}
\label{lem:left:right}
Let $R$ be a unital ring, and suppose that $x \in R$ admits both a left and a right inverse. Then the inverses coincide. In particular, the following are equivalent:
\begin{enumerate}
\item $R$ is directly finite; that is, every element with a left inverse is a unit.
\item For all $x, y \in R$, if $xy = 1_R$ then $yx = 1_R$
\end{enumerate}
\end{lemma}

\begin{proof}
Suppose that $xy = yz = 1_R$. Then $x = x(yz) = (xy)z = z$.
\end{proof}

This immediately implies the following:

\begin{lemma}
\label{lem:DFsubobject}
If $S$ is a unital subring of $R$ and $R$ is directly finite, then $S$ is directly finite. In particular, if $S$ is a unital subring of $R$, $d' \geq d$ and $\mathbb{M}_{d'}(R)$ is directly finite, then $\mathbb{M}_d(S)$ is directly finite.
\end{lemma}

As mentioned in the Introduction, the two conjectures are equivalent. This fact was already known to Passman (see the Mathscinet review \cite{review}), but first appeared in print in \cite{df:dykemajuschenko}:

\begin{theorem}
\label{thm:df:equivalence}
Let $\mathbb{F}$ be a field and let $\Gamma$ be a group. Then $\mathbb{F}[\Gamma]$ is stably finite if and only if $\mathbb{F}[\Gamma \times H]$ is directly finite for every finite group $H$.
Therefore, Kaplansky's direct finiteness conjecture is equivalent to Kaplansky's stable finiteness conjecture.
\end{theorem}

The most important fact about these conjectures is that they hold in characteristic $0$:

\begin{theorem}[{\cite[p. 122]{kaplansky}}]
\label{thm:df:char0}
Let $\mathbb{F}$ be a field of characteristic $0$ and let $\Gamma$ be any group. Then $\mathbb{F}[\Gamma]$ is stably finite.
\end{theorem}

In fact, Kaplansky's original formulation of the conjecture \cite[p. 123]{kaplansky} only concerns stable finiteness over fields of positive characteristic. Let us record a corollary that will be useful for our purposes:

\begin{corollary}
\label{cor:df:Z}
Let $\Gamma$ be a group. Then $\mathbb{Z}[\Gamma]$ is stably finite.
\end{corollary}

\begin{proof}
This follows form Theorem \ref{thm:df:char0} and Lemma \ref{lem:DFsubobject}, by embedding $\mathbb{Z}[\Gamma]$ as a subring of $\mathbb{Q}[\Gamma]$.
\end{proof}

Now let us mention some groups that are known to satisfy these conjectures. The main example is that of \emph{sofic groups}:

\begin{definition}\label{defi:sofic}
Let $\Gamma$ be a group. Let $0 < \varepsilon < 1$, $n \geq 1$ and let $F \subset \Gamma$ be a finite subset of $\Gamma$. A map $\varphi \colon F \to S_n$ is called an \emph{$(F, \varepsilon)$-approximation} if the following two conditions hold:
\begin{enumerate}
    \item For all $g, h \in F$, if $gh \in F$, then it holds $d_H(\varphi(gh), \varphi(g)\varphi(h)) \leq \varepsilon$;
    \item For all $1_\Gamma \neq g \in F$ it holds $d_H(\varphi(g), \id) \geq (1 - \varepsilon)$.
\end{enumerate}
Here $d_H$ denotes the \emph{Hamming distance} on $S_n$, that is $d_H(\sigma, \tau) = \frac{1}{n} \# \{ i : \sigma(i) \neq \tau(i) \}$.

The group $\Gamma$ is said to be \emph{sofic} if for every $0 < \varepsilon < 1$ and every finite subset $F \subset \Gamma$ there exists an $(F, \varepsilon)$-approximation into some $S_n$.
\end{definition}

The class of sofic groups was introduced by Gromov \cite{gromov} and Weiss \cite{weiss} as a large class of groups for which Gottschalk's surjunctivity conjecture could be proven. We will discuss this more in Subsection \ref{s:additive}. The main examples of sofic groups are amenable and residually finite groups, and several constructions (subgroups, directed unions, marked limits, extensions by amenable groups) preserve soficity: We refer the reader to \cite{cell} for more detail. To this day, there is no known example of a non-sofic group.

\begin{theorem}[\cite{df:elekszabo}]
\label{thm:sofic}

Sofic groups satisfy Kaplansky's stable finiteness conjecture.
\end{theorem}

Theorem \ref{thm:sofic} was proven again by different methods in \cite{df:ceccherinicoornaert} and then in \cite{linearsofic}. 
In Subsection \ref{s:additive}, we will provide yet another proof. Let us however mention that we will only be dealing with group rings over fields, but the stable finiteness for group rings of sofic groups holds more generally over division rings \cite{df:elekszabo}, and in fact over Noetherian rings \cite{df:noetherian}.

\medskip

Another relevant class of examples is that of groups with the \emph{unique product property}, or more succintly \emph{UPP groups}. We say that $\Gamma$ is a UPP group if for every pair of nonempty finite subsets $S, T \subset \Gamma$ there exists an element $g \in \Gamma$ that can be \emph{uniquely} expressed as the product of an element of $S$ and an element of $T$. Relevant examples of UPP groups are \emph{left orderable groups}, that is, groups admitting a total order that is invariant by left translation. It can be shown that for every field $\mathbb{F}$ and every UPP group $\Gamma$, the group ring $\mathbb{F}[\Gamma]$ is directly finite (the proof is the same as the fact that UPP groups have no nontrivial units, see e.g. \cite[Chapter 13]{passman}). However, we do not know whether UPP groups, or even left orderable groups, satisfy Kaplansky's stable finiteness conjecture. Theorem \ref{thm:df:equivalence} does not help in this case, since no group with torsion is UPP.

\begin{question}
Do UPP groups satisfy Kaplansky's stable finiteness conjecture? Do left orderable groups?
\end{question}

For left orderable groups, the following special case is known:

\begin{proposition}
\label{prop:biorderable}
    Bi-orderable groups satisfy Kaplansky's stable finiteness conjecture.
\end{proposition}

A group is \emph{bi-orderable} if it admits a total order that is invariant under both left and right translation.

\begin{proof}
    This follows from the fact that if $\Gamma$ is bi-orderable and $\mathbb{F}$ is a field, then $\mathbb{F}[\Gamma]$ embeds into a division ring \cite{skew1, skew2}.
\end{proof}

An intermediate notion between bi-orderability and left orderability is \emph{local indicability}: A group is locally indicable if every finitely generated subgroup surjects onto $\mathbb{Z}$, and this is equivalent to the existence of a left invariant order satisfying a weak version of bi-invariance \cite{conrad, brodskii}. Group rings of locally indicable groups embed into division rings in characteristic $0$ \cite{jaikin}, but to our knowledge this is still open in positive characteristic, which is the case of relevance for the stable finiteness conjecture.

\medskip

Using these known results, we can deduce Corollary \ref{intro:cor:solution} from the Introduction, assuming Theorem \ref{intro:thm:main}.

\begin{corollary}
\label{cor:solution}
    Let $\Gamma$ be a finitely generated group, and let $A$ be a finitely generated abelian group. Suppose that one of the following holds:
    \begin{enumerate}
        \item $A$ is torsion-free;
        \item $\Gamma$ is sofic;
        \item $\Gamma$ is bi-orderable;
        \item $A$ is cyclic and $\Gamma$ has the unique product property.
    \end{enumerate}
    Then $A \wr \Gamma$ is Hopfian if and only if $\Gamma$ is Hopfian.
\end{corollary}

\begin{proof}
    Let $A \coloneqq A_0 \oplus \bigoplus A_p$, where $A_0$ is free abelian, and $A_p$ is a $p$-group which decomposes further as $A_p \coloneqq (\mathbb{Z}/p\mathbb{Z})^{d^p_1} \oplus \cdots \oplus (\mathbb{Z}/p^m\mathbb{Z})^{d^p_m}$, where $d^p_i \in \mathbb{N}$. By Theorem \ref{intro:thm:main}, it suffices to show that, under each of the four assumptions above, the ring $\mathbb{M}_{\max_i(d^p_i)}(\mathbb{F}_p[\Gamma])$ is directly finite.

    When $A$ is torsion-free, there is nothing to check. When $\Gamma$ is sofic, this follows from Theorem \ref{thm:sofic}, and when $\Gamma$ is bi-orderable, this follows from Proposition \ref{prop:biorderable}. When $A$ is cyclic, each $d_i^p = 1$, and so this amounts to direct finiteness of $\mathbb{F}_p[\Gamma]$, which holds for UPP groups \cite[Chapter 13]{passman}.
\end{proof}

\subsection{Local embeddings}

The notion of a local embedding goes back to the work of Mal'cev \cite{malcev} although it is commonly attributed to \cite{LEF}. Let us recall the definition for rings:

\begin{definition}
\label{defi:LE}
Let $R$ be a ring and $\mathcal{S}$ a class of rings. We say that $R$ is \emph{locally embeddable into $\mathcal{S}$} if for every finite set $F \subset R$ there exists a map $f \colon F \to S \in \mathcal{S}$ such that:

\begin{enumerate}   

\item For all $x, y \in F$, if $x + y \in F$, then it holds $f(x + y) = f(x) + f(y)$;
\item For all $x, y \in F$, if $xy \in F$, then it holds $f(xy) = f(x)f(y)$;
\item $f|_F$ is injective.

\end{enumerate}

If $R$ has an identity, we also require that each $S \in \mathcal{S}$ has an identity, and $f(1_R) = 1_S$ (if $1_R \in F)$.
Such a map $f$ is called a \emph{local embedding}
of $F$ into $S$.
\end{definition}

This is an exact version of the notion of approximation we used to introduce sofic groups.

\begin{example} \label{ex:SubringLE}
If $f \colon R \rightarrow S$ is an injective
homomorphism of rings,
then the restriction of $f$ to any finite subset
of $R$ is a local embedding.
Thus $R$ locally embeds into the class
$\lbrace S \rbrace$.
\end{example}

\begin{remark}
Let $R$ be a unital ring and $\mathcal{S},\mathcal{T}$
be classes of unital rings.
If $R$ locally embeds into $\mathcal{S}$ and
every $S \in \mathcal{S}$ locally embeds into $\mathcal{T}$,
then $R$ locally embeds into $\mathcal{T}$.
In particular, if $f \colon R \rightarrow S$
is an injective homomorphism of rings and
$S$ is locally embeddable into $\mathcal{T}$,
then taking $\mathcal{S} = \lbrace S \rbrace$
in the above and using Example \ref{ex:SubringLE},
we have that local embeddability into
a class $\mathcal{T}$ is preserved under taking
subrings.
\end{remark}

The main result of this subsection is Proposition \ref{intro:prop:LE} from the Introduction, which we recall for the reader's convenience:

\begin{proposition}
\label{prop:LE}
Let $\mathbb{F}$ be a field of characteristic $p > 0$. Then $\mathbb{F}$ is locally embeddable into finite fields of characteristic $p$. In particular, $\mathbb{F}$ is locally embeddable into matrix algebras over $\mathbb{F}_p$, namely $\{ \mathbb{M}_d(\mathbb{F}_p) : d \geq 1 \}$.
\end{proposition}

Before moving on to the proof, let us record the following corollary (Corollary \ref{intro:cor:kaplansky:Fp} from the Introduction):

\begin{corollary}
\label{cor:kaplansky:Fp}
Let $\Gamma$ be a group; let $p$ be a prime
and let $\mathbb{F}$ be a field of characteristic $p$.
Then $\mathbb{F}[\Gamma]$ is stably finite if and only if
$\mathbb{F}_p [\Gamma]$ is stably finite.
Thus Kaplansky's stable finiteness conjecture
over $\mathbb{F}$ implies Kaplansky's stable finiteness conjecture over all fields of characteristic $p$.
\end{corollary}

\begin{proof}
For the first statement,
if $\mathbb{F}[\Gamma]$ is stably finite,
then the subring
$\mathbb{F}_p[\Gamma]$ is stably finite
by Lemma \ref{lem:DFsubobject}.
Conversely suppose that $\mathbb{F}[\Gamma]$ is not stably finite.
Then by Theorem \ref{thm:df:equivalence} there
exists a finite group $H$ such that
$\mathbb{F}[\Gamma \times H]$ is not directly finite.
Write $\Delta = \Gamma \times H$, so that
there exist $x, y \in \mathbb{F}[\Delta]$ such that $xy = 1_{\mathbb{F}[\Delta]}$ but $yx \neq 1_{\mathbb{F}[\Delta]}$. 
Let $E \subset \mathbb{F}$ be the union of all coefficients of $x, y, xy$ and $yx$ 
(a finite set containing $0$ and $1$). 
Let $n \in \mathbb{N}$ be such that there exists a subset of $\Delta$ of size $n$ 
spanning both $x$ and $y$. 
Let:
\[F \coloneqq \left\{ \sum_{i=1} ^n \lambda_i \mu_i : \lambda_i , \mu_i \in E \right\} \subseteq \mathbb{F}\]
and let $f \colon F \to \mathbb{M}_d(\mathbb{F}_p)$ be a local embedding: The latter exists by Proposition \ref{prop:LE}. We extend $f$ to a map $F[\Delta] \to \mathbb{M}_d(\mathbb{F}_p)[\Delta]$ - where $F[\Delta]$ denotes the subset of elements of  $\mathbb{F}[\Delta]$ such that all images belong to $F$ - as follows:
$$f \left( \sum\limits_{\delta \in \Delta} z \cdot \delta \right) = \sum\limits_{\delta \in \Delta} f(z) \cdot \delta.$$
Then by construction of $F$ and the local embedding property, 
$f(x)f(y) = f(xy) = f(1_{\mathbb{F}[\Gamma]}) = I_d$, while $f(y) f(x) = f(yx) \neq I_d$ by injectivity. Therefore $\mathbb{M}_d(\mathbb{F}_p)[\Delta]$ is not directly finite. 
But $\mathbb{M}_d(\mathbb{F}_p)[\Delta]$ is naturally isomorphic to $\mathbb{M}_d(\mathbb{F}_p[\Delta])$, so $\mathbb{F}_p[\Delta]$ is not stably finite. By Theorem \ref{thm:df:equivalence},
there exists a finite group $K$ such that
$\mathbb{F}_p [\Delta \times K]$ is not directly finite.
Since $\Delta \times K \cong \Gamma \times (H \times K)$
and $H \times K$ is a finite group,
one more application of Theorem \ref{thm:df:equivalence}
yields that $\mathbb{F}_p [\Gamma]$
is not stably finite.
\end{proof}

As a consequence, we have the following neat reformulation of the stable finiteness conjecture. 

\begin{corollary}
Kaplansky's stable finiteness conjecture holds iff for every group $\Gamma$ and every prime $p$, 
$\mathbb{F}_p[\Gamma]$ is directly finite. 
\end{corollary}

\begin{proof}
One direction is clear. Converesly, for fixed $p$, 
direct finiteness of $\mathbb{F}_p[\Gamma]$ for every 
group $\Gamma$ implies stable finiteness of $\mathbb{F}_p[\Gamma]$ for every 
group $\Gamma$, by Theorem \ref{thm:df:equivalence}. 
Corollary \ref{cor:kaplansky:Fp} then implies stable finiteness of $\mathbb{F}[\Gamma]$ 
for every field $\mathbb{F}$ of characteristic $p$ and every group $\Gamma$. 
Finally, the characteristic zero case follows from Theorem \ref{thm:df:char0}. 
\end{proof}

\begin{remark}
The first statement of Proposition \ref{prop:LE} allows to show, with a similar argument, that the zero-divisor and idempotent conjectures over field of characteristic $p$ reduce to the case of finite fields of characteristic $p$. See \cite[Section 3.4.3]{malman} for an argument covering the zero-divisor conjecture. But to reduce to $\mathbb{F}_p$ specifically, we crucially exploit the flexibility of changing the degree of the matrix algebras, and this is allowed by the stable finiteness conjecture only.
\end{remark}

We now turn to the proof of Proposition \ref{prop:LE}. 
First, we explain how the second statement in Proposition \ref{prop:LE} follows from the first.
Recall that if $A$ is a finite-dimensional
algebra over a field $\mathbb{K}$,
then $A$ naturally acts on itself
faithfully by linear transformations,
so that a choice of $\mathbb{K}$-basis for $A$
induces an embedding of $A$ as a subring of
$\mathbb{M}_d(\mathbb{K})$,
where $d=\dim_{\mathbb{K}}(A)$.
In particular we have the following:

\begin{lemma} \label{lem:FFMatEmb}
If $\mathbb{F}$ is a finite field of characteristic $p$,
then $\mathbb{F}$ is isomorphic to a subring of
$\mathbb{M}_{d} (\mathbb{F}_p)$ for some $d \geq 1$.
\end{lemma}

The key to the proof of Proposition \ref{prop:LE} is
\emph{Zariski's Lemma} (a close cousin of Hilbert's Nullstellensatz,
see \cite{AtiyahMcD} Proposition 7.9).

\begin{theorem} \label{thm:Zariski}
Let $\mathbb{E}$ be a subfield of $\mathbb{F}$.
Suppose that $\mathbb{F}$ is finitely generated as an associative
algebra over $\mathbb{E}$.
Then $\lvert \mathbb{F} : \mathbb{E} \rvert$ is finite. 
\end{theorem}

\begin{proof}[Proof of Proposition \ref{prop:LE}]
Let $F$ be a finite subset of $\mathbb{F}$.
Let $R$ be the subring of $\mathbb{F}$ generated by
$A = F \cup \lbrace (x-y)^{-1} : x,y \in F , x \neq y \rbrace$,
and let $M \triangleleft R$ be a maximal (not necessarily nonzero) ideal in $R$.
Then $R/M$ is a field, which as an $\mathbb{F}_p$-algebra
is generated by the finite set $A+M$.
Thus by Theorem \ref{thm:Zariski} $R/M$ is a finite field, 
and hence by Lemma \ref{lem:FFMatEmb} is isomorphic to a 
subring of some $\mathbb{M}_{d} (\mathbb{F}_p)$. 
Finally, the quotient homomorphism $R \rightarrow R/M$
restricts to an injection of $F$,
since by construction of $A$, whenever $x , y \in F$ with $x \neq y$,
$x-y$ is a unit in $R$, so does not lie in $M$. 
\end{proof}

\subsection{Additive cellular automata}
\label{s:additive}

In this subsection, we apply Corollary \ref{cor:kaplansky:Fp} to obtain yet another equivalent version of Kaplansky's stable finiteness conjecture, proving Theorem \ref{intro:thm:A:surj} and Corollary \ref{intro:cor:surjunctive} from the Introduction. We start by introducing the relevant definitions.

\begin{definition}\label{defi:automata}
Let $\Gamma$ be a group, and let $F$ be a finite set, called the \emph{alphabet}. Let $F^{\Gamma}$ be the set $\prod\limits_{\gamma \in \Gamma} F$ endowed with the usual left action of $\Gamma$ and the prodiscrete topology. A \emph{cellular automaton over $\Gamma$} is a continuous $\Gamma$-equivariant map $f \colon F^\Gamma \to F^\Gamma$.

If $F$ is a finite-dimensional vector space over a field $\mathbb{K}$, and $f$ is a linear map, then $f$ is called a \emph{linear cellular automaton}.

If $F$ is a finite abelian group, and $f$ is a homomorphism, then $f$ is a called an \emph{additive cellular automaton}.
\end{definition}

The classical definition of cellular automaton is usually given in terms of memory sets and local defining maps \cite[Definition 1.4.1]{cell}, however the above definition is equivalent in case the underlying alphabet $F$ is finite \cite[Theorem 1.8.1]{cell}, which will be our case of interest. There is an immense literature on cellular automata; we refer the reader to the book \cite{cell} for a group-theoretic viewpoint, which is most relevant for our purposes. Linear cellular automata over groups have been extensively studied by Ceccherini-Silberstein and Coornaert \cite[Chapter 8]{cell}. On the other hand, while the study of additive cellular automata is very well-developed in theoretical computer science (see \cite{additive} for a survey), to our knowledge an analysis parallel to the one of linear cellular automata over groups is absent from the literature.

\medskip

One of the fundamental problems in the theory of cellular automata is to understand so-called \emph{Garden of Eden states}, that is elements that lie outside the image of a given cellular automaton. This leads naturally to the study of the classes of groups for which Garden of Eden states do not exist, under the natural assumption of injectivity (which in certain classical settings is equivalent to reversibility):

\begin{definition}\label{defi:surjuctive}
A group $\Gamma$ is said to be \emph{surjunctive} if every injective cellular automaton over $\Gamma$ (with alphabet a finite set) is surjective. It is said to be \emph{$L$-surjunctive} if every injective linear cellular automaton over $\Gamma$ is surjective. It is said to be \emph{$A$-surjunctive} if every injective additive cellular automaton is surjective.
\end{definition}

The property of surjunctivity was introduced by Gottschalk in \cite{gottschalk}, who asked whether all groups are surjunctive: This is now known as the \emph{Gottschalk surjunctivity conjecture}. The class of sofic groups was introduced by Gromov \cite{gromov} and Weiss \cite{weiss} precisely in this context, for proving at once surjunctivity of many groups, including amenable and residually finite groups.
As mentioned above, the notion of $A$-surjunctivity seems to be absent from the literature.
The problem of $L$-surjunctivity has been studied extensively by Ceccherini-Silberstein and Coornaert \cite{df:ceccherinicoornaert, csc1, csc3, csc2}. In particular, they proved the following striking connection to Kaplansky's stable finiteness conjecture:

\begin{theorem}[\cite{df:ceccherinicoornaert}]
\label{thm:csc}
Let $\Gamma$ be a group. Then $\Gamma$ is $L$-surjunctive if and only if $\Gamma$ satisfies Kaplansky's stable finiteness conjecture. More precisely, given a $d$-dimensional vector space $V$ over a field $\mathbb{K}$, the following are equivalent:
\begin{enumerate}
    \item Every injective linear cellular automaton $V^\Gamma \to V^\Gamma$ is surjective.
    \item The ring $\mathbb{M}_d(\mathbb{K}[\Gamma])$ is directly finite.
\end{enumerate}
\end{theorem}

Here the field $\mathbb{K}$, and thus the vector space $V$, need not be finite, and in that case the definition of cellular automaton assumed in the statement is the classical one \cite[Definition 1.4.1]{cell}. However, thanks to our results of the previous subsection, we deduce that for the characterization it suffices to look at finite fields:

\begin{corollary}
\label{cor:L:finite}

Let $\Gamma$ be a group. Then $\Gamma$ is $L$-surjunctive if and only if $\Gamma$ satisfies Kaplansky's stable finiteness conjecture over $\mathbb{F}_p$, for all primes $p$.
\end{corollary}

\begin{proof}
This follows directly by combining Theorem \ref{thm:csc} and the fact that Kaplansky's stable finiteness conjecture reduces to the fields $\mathbb{F}_p$, by Theorem \ref{thm:df:char0} and Corollary \ref{cor:kaplansky:Fp}.
\end{proof}

Our goal in this subsection is to show that the notions of $L$-surjunctivity and $A$-surjunctivity coincide:

\begin{theorem}
\label{thm:A:surj}

Let $\Gamma$ be a group. Then $\Gamma$ is $A$-surjunctive if and only if it is $L$-surjunctive.
\end{theorem}

Our proof will use Theorem \ref{thm:csc} as a base case for an induction argument, which is a mild version of the more involved induction argument that we will use in the next section.

\begin{proof}[Proof of Theorem \ref{thm:A:surj}]
If $\Gamma$ is $A$-surjunctive, then every injective additive cellular automaton $(\mathbb{F}_p^d)^\Gamma \to (\mathbb{F}_p^d)^\Gamma$ is surjective. This applies in particular to linear cellular automata, so $\Gamma$ is $L$-surjunctive.

Conversely, suppose that $\Gamma$ is $L$-surjunctive, and let $A$ be a finite abelian group, which we decompose as a sum of $p$-groups $A = \bigoplus_p A_p$. Then $A^\Gamma$ decomposes as $\bigoplus_p A_p^\Gamma$, and this decomposition is $\Gamma$-equivariant, where $\Gamma$ acts on the direct sum by shifting diagonally. Let $f \colon A^\Gamma \to A^\Gamma$ be an injective additive cellular automaton. Then, as in the proof of Lemma \ref{lem:AxB:Hopfian}, the image of $A_p^\Gamma$ is a $p$-group, so it belongs to $A_p^\Gamma$. It follows that $f$ restricts to injective additive cellular automata $A_p^\Gamma \to A_p^\Gamma$, and therefore we reduce to the case in which $A$ is a $p$-group.

We proceed by induction on the exponent of $A$. Let $Q$ be the subgroup of elements of order dividing $p$. Then $Q$ is a finite abelian group of exponent $p$, and therefore it is the additive group of a finite-dimensional vector space over $\mathbb{F}_p$. Moreover, the image of $Q^\Gamma$ under $f$ also has exponent $p$, which implies that $f$ restricts to an injective additive cellular automaton $Q^\Gamma \to Q^\Gamma$. This is automatically linear, and since $Q$ is a finite-dimensional $\mathbb{F}_p$-vector space, by the assumption on $L$-surjunctivity we have $f(Q^\Gamma) = Q^\Gamma$, and moreover by injectivity $f^{-1}(Q^\Gamma) = Q^{\Gamma}$ as well. Since $Q^\Gamma$ is closed and $\Gamma$-invariant, it follows that $f$ induces an injective additive cellular automaton $(A/Q)^\Gamma \to (A/Q)^\Gamma$. The exponent of $A/Q$ is strictly smaller than the exponent of $A$, and thus we conclude by induction. 
\end{proof}

In particular we obtain Corollary \ref{intro:cor:surjunctive} from the Introduction:

\begin{corollary}
\label{cor:surjunctive}

Surjunctive groups satisfy Kaplansky's stable finiteness conjecture. In particular, sofic groups satisfy Kaplansky's stable finiteness conjecture.
\end{corollary}

\begin{proof}
Surjunctive groups are clearly $A$-surjunctive, thus $L$-surjunctive by Theorem \ref{thm:A:surj}. We conclude by Theorem \ref{thm:csc}. The last statement follows from the Gromov--Weiss Theorem \cite{gromov, weiss}.
\end{proof}

\begin{remark}
Our argument went through $A$-surjunctivity, which in turn used the result for $L$-surjunctivity from \cite{df:ceccherinicoornaert}. However, the corollary for surjunctive groups is really an elementary consequence of Corollary \ref{cor:kaplansky:Fp}: Indeed, it was noticed already in \cite{df:elekszabo} that surjunctive groups satisfy the direct finiteness conjecture over group rings, and virtually surjunctive groups are surjunctive \cite{arzhantsevagal}.
\end{remark}

Corollary \ref{cor:surjunctive} was recently independently proven by Phung \cite{df:phung} using methods from the theory of symbolic algebraic varieties. Still, the blueprint of his proof is similar to ours, in that it goes through the surjunctivity property for an algebro-geometric category of cellular automata.

\medskip

Let us end this section by pointing out that surjunctivity and Hopficity are in some sense dual to each other. In one case, injectivity implies surjectivity, and in the other case surjectivity implies injectivity. The relation between the two only occurs via direct finiteness, which is a symmetric property. In fact, in the results of the next section 
(for instance in the proof of Theorem \ref{PGrpStabFinThm}) 
one direction is elementary and follows a similar approach as the proof that surjunctivity implies stable finiteness; but the other direction will need an ad-hoc approach, which is moreover only possible thanks to the preliminary work in Section \ref{s:wreath}. 

\section{Wreath products and stable finiteness}
\label{s:main}

In this section we prove Theorem \ref{AbBaseIffThm},
which combined with Proposition \ref{prop:decomposition} and Corollary \ref{cor:solutionfreeabelian} gives Theorem
\ref{intro:thm:main} from the Introduction.
Using Corollary \ref{cor:kaplansky:Fp}
we then deduce Theorem \ref{intro:thm:main:general} (Theorem \ref{thm:main:general} below).

The following observation appears as Exercise 8.13 in \cite{cell}.
Since (a special case of) this observation is
essential for our purposes,
we give a self-contained proof.
For $R$ a ring and $d \geq 1$,
we write elements of $R^d$ as row-vectors
(so that matrices act on the right).

\begin{lemma} \label{lem:HopfModulesDFMatrix}
Let $R$ be a unital ring and let $d \geq 1$.
Then $R^d$ is a Hopfian left $R$-module
if and only if the ring $\mathbb{M}_d (R)$
is directly finite.
\end{lemma}

\begin{proof}
First suppose that $\mathbb{M}_d ( R )$
is directly finite
and that $\varphi \colon R^d \rightarrow R^d$ is
a surjective left $R$-module endomorphism.
Let $\mathcal{B}$ be the standard free basis for $R^d$.
For each $b \in \mathcal{B}$,
let $c_b \in \varphi^{-1} (b)$.
Let $\psi \colon R^d \rightarrow R^d$ be the (unique) left
$R$-module endomorphism of $R^d$ given by
$\psi(b) = c_b$ for all $b \in \mathcal{B}$.
Then $\varphi \circ \psi = \id_{R^d}$.
Since $\End_{R}(R^d) \cong \mathbb{M}_d (R)$,
we deduce $\psi \circ \varphi = \id_{R^d}$ also,
and $\varphi$ is an isomorphism.

Conversely suppose that
$X,Y \in \mathbb{M}_d (R)$
satisfy $XY=I_d$ but $YX\neq I_d$.
Define $\varphi \colon R^d \rightarrow R^d $
by $\varphi (v) = vY$.
Then $\varphi$ is clearly a left $R$-module homomorphism.
Moreover $v = \varphi (vX)$, so $\varphi$ is surjective.
Since $YX\neq I_d$, there exists $v \in R^d$ such that $v YX \neq v$.
Let $0 \neq u = vYX-v \in R^d$.
Then $uY = 0$ so $u \in \ker (\varphi)$,
so $\varphi$ is not injective, and $R^d$ is not Hopfian.
\end{proof}

Next, we examine the special case in which the base group is a power of a fixed cyclic group, which will serve as the basis of an induction in the general case. We use the convention that $\mathbb{Z}/n\mathbb{Z} = \mathbb{Z}$ when $n = 0$.

\begin{theorem} \label{PGrpStabFinThm}
Let $\Gamma$ be a finitely generated Hopfian group;
let $n \geq 0$, and let $d \geq 1$.
Then
$\big((\mathbb{Z}/n\mathbb{Z})^d\big) \wr \Gamma$ is Hopfian
if and only if $\mathbb{M}_d \big( (\mathbb{Z}/n\mathbb{Z}) [\Gamma] \big)$ is directly finite.
\end{theorem}

\begin{proof}
If $\Gamma$ is finite, then $\big((\mathbb{Z}/n\mathbb{Z})^d\big) \wr \Gamma$ is Hopfian, being finitely generated and residually finite by \cite{grunberg} (note that $\big((\mathbb{Z}/n\mathbb{Z})^d\big) \wr \Gamma$ is not necessarily finite, since $n$ is allowed to be equal to $0$).
Moreover $(\mathbb{Z}/n\mathbb{Z})[\Gamma]$ is stably finite: This is immediate when $n \geq 1$ since finite rings are stably finite, and it is Corollary \ref{cor:df:Z} in case $n = 0$. Therefore we may assume that $\Gamma$ is infinite,
which allows us to use the results from Section \ref{s:wreath}.

Let $R = (\mathbb{Z}/n\mathbb{Z}) [\Gamma]$.
Then $(\mathbb{Z}/n\mathbb{Z})^d [\Gamma]$
is naturally a left $R$-module.
Indeed, it is isomorphic to $R^d$ as a left $R$-module:
Given $f \in (\mathbb{Z}/n\mathbb{Z})^d [\Gamma]$,
write:
\begin{center}
$f(g) = \big( f(g)_1 , \ldots , f(g)_d \big)$
\end{center}
for $g \in \Gamma$
(with $f(g)_i \in \mathbb{Z}/n\mathbb{Z}$),
and identify $f$ with $(f_1 , \ldots f_d) \in R^d$,
where $f_i \in R = (\mathbb{Z}/n\mathbb{Z}) [\Gamma]$
is given by $f_i(g) = f(g)_i$.
Under this identification,
we have
$\big((\mathbb{Z}/n\mathbb{Z})^d\big) \wr \Gamma
\cong R^d \rtimes \Gamma$,
where the group operation is given by:
\begin{center}
$(v_1 g_1) (v_2,g_2) = (v_1 + g_1 \cdot v_2,g_1 g_2)$
\end{center}
(where $g_1 \cdot v_2$ denotes the left action of
$\Gamma$ on $R^d$ under the left
$(\mathbb{Z}/n\mathbb{Z}) [\Gamma]$-module structure of $R^d$).
For the remainder of the proof we work with
$R^d \rtimes \Gamma$.

First suppose that $R^d \rtimes \Gamma$ is non-Hopfian,
and that $\varphi \colon R^d \rtimes \Gamma \to R^d \rtimes \Gamma$
is a non-injective group epimorphism.
Let $\psi \colon R^d \rtimes \Gamma \rightarrow R^d \rtimes \Gamma$
be as in Proposition \ref{prop:basic:morphism}; that is $\psi(v, \gamma) = (\varphi(v), \gamma)$, and $\psi$ is also a non-injective group epimorphism.
Then the homomorphism relation implies:
\begin{center}
$\big( \varphi (g \cdot v),g\big)
= \psi (g \cdot v,g)
= \psi (0,g) \psi(v,1_{\Gamma})
= \big( g \cdot \varphi (v),g\big)$
\end{center}
for all $v \in R^d$ and $g \in \Gamma$,
so the restriction of $\varphi$
to $R^d$ is in fact a left $R$-module homomorphism,
which by Proposition \ref{prop:basic:morphism} is
surjective but not injective.
In other words, $R^d$ is a non-Hopfian left $R$-module,
and by Lemma \ref{lem:HopfModulesDFMatrix},
$\mathbb{M}_d(R)$ is not directly finite.

Conversely suppose that $\mathbb{M}_d(R)$
is not directly finite,
so that by Lemma \ref{lem:HopfModulesDFMatrix}
there exists a surjective left $R$-module homomorphism
$\psi \colon R^d \to R^d$
which is not injective.
Define $\varphi \colon R^d \rtimes \Gamma \rightarrow R^d \rtimes \Gamma$
by $\varphi (v,g) = (\psi(v),g)$.
Then $\varphi$ is a homomorphism of groups,
which is surjective but not injective.
\end{proof}

This already implies the full solution to our problem for free abelian groups:

\begin{corollary}
\label{cor:solutionfreeabelian}

Let $A$ be a free abelian group of finite rank, and let $\Gamma$ be a finitely generated group. Then $\Gamma$ is Hopfian if and only if $A \wr \Gamma$ is Hopfian.
\end{corollary}

\begin{proof}
Let $A = \mathbb{Z}^d$.
If $A \wr \Gamma$ is Hopfian,
then by Lemma \ref{lem:G:Hopfian},
so too is $\Gamma$.
Conversely, if $A \wr \Gamma$ is non-Hopfian,
then Theorem \ref{PGrpStabFinThm} (applied with $n=0$)
implies that $\mathbb{M}_d (\mathbb{Z}[\Gamma])$
is not directly finite, contradicting
Corollary \ref{cor:df:Z}.
\end{proof}

Thanks to Proposition \ref{prop:decomposition}, we are left to understand when $P \wr \Gamma$ is Hopfian, where $P$ is a finite abelian $p$-group.

\begin{theorem} \label{AbBaseIffThm}
Let $\Gamma$ be a finitely generated Hopfian group.
Let $P = (\mathbb{Z}/p\mathbb{Z})^{d_1} \oplus \cdots \oplus (\mathbb{Z}/p^m\mathbb{Z})^{d_m}$ be a finite abelian $p$-group.
Then $P \wr \Gamma$ is Hopfian if and only if
$\mathbb{M}_{\max_i (d_i)} \big( \mathbb{F}_p [\Gamma] \big)$
is directly finite.
\end{theorem}

Note that Theorems \ref{PGrpStabFinThm} and \ref{AbBaseIffThm} combine to immediately yield 
the following conclusion about directly finite rings. 

\begin{corollary}
Let $\Gamma$ be a finitely generated Hopfian group and let $d,m \geq 1$. 
Then $\mathbb{M}_d \big( \mathbb{F}_p [\Gamma] \big)$ 
is directly finite iff 
$\mathbb{M}_d \big( (\mathbb{Z}/p^m\mathbb{Z}) [\Gamma] \big)$ 
is directly finite. 
\end{corollary}

As a first step towards Theorem \ref{AbBaseIffThm}, 
we have the following reduction, which is essentially an application of Hensel's Lemma. 

\begin{lemma} \label{lem:Hensel}
Let $\Gamma$ be a group and let $d , m \geq 1$. 
If $\mathbb{M}_d \big( (\mathbb{Z}/p^{m+1}\mathbb{Z}) [\Gamma] \big)$
is directly finite, then so is
$\mathbb{M}_d \big( (\mathbb{Z}/p^m\mathbb{Z}) [\Gamma] \big)$.
\end{lemma}

\begin{proof}
Let $X,Y \in \mathbb{M}_d \big( (\mathbb{Z}/p^m\mathbb{Z}) [\Gamma] \big)$ 
and suppose $XY = I_d$. 
Let $\tilde{X},\tilde{Y} \in \mathbb{M}_d \big( \mathbb{Z} [\Gamma] \big)$ 
with $\tilde{X} \equiv X,\tilde{Y} \equiv Y \mod p^m$, 
so that there exists $Z \in \mathbb{M}_d \big( \mathbb{Z} [\Gamma] \big)$ such that 
$\tilde{X} \tilde{Y} = I_d + p^m Z$. 
Set $\overline{X} = \tilde{X} , \overline{Y} = \tilde{Y} (I_d - p^m Z)$, 
so that $\overline{X} \overline{Y} = I_d - p^{2m} Z \equiv I_d \mod p^{m+1}$. 
By direct finiteness of $\mathbb{M}_d \big( (\mathbb{Z}/p^{m+1}\mathbb{Z}) [\Gamma] \big)$, 
we have $\overline{Y} \overline{X} \equiv I_d \mod p^{m+1}$, 
but $\overline{X} \equiv X , \overline{Y} \equiv Y \mod p^m$, 
so $YX = I_d$ also. 
\end{proof}

\begin{notation} \label{BlockMxNotn}
For $\mathbf{d} = (d_1 , \ldots , d_m) \in \mathbb{N}^m$;
$D = d_1 + \cdots + d_m$,
and $R$ a unital ring, let $B_{\mathbf{d}} (R) \leq \mathbb{M}_D (R)$
be the subring given by:
\begin{equation*}
B_{\mathbf{d}} (R) \coloneqq
\left\{ \left( \begin{array}{cccc}
A_1 & \ast & \cdots & \ast \\
 0 & A_2 & \ddots & \vdots \\
 \vdots & \ddots & \ddots & \ast \\
 0 & \cdots & 0 & A_m
\end{array} \right) : A_i \in \mathbb{M}_{d_i} (R) \right\}.
\end{equation*}
\end{notation}

\begin{lemma} \label{lem:BlockStrongLeftUnit}
Keep the terminology from Notation \ref{BlockMxNotn}.
Suppose that $\mathbb{M}_{\max_i (d_i)} (R)$ is directly finite.
Let $X \in B_{\mathbf{d}} (R)$ and $Y \in \mathbb{M}_D (R)$
be such that $YX = I_D$.
Then $Y \in B_{\mathbf{d}} (R)$.
\end{lemma}

\begin{proof}
We proceed by induction on $m$;
the base case $m=1$ is vacuous.
Write $X = (X_{i,j})_{1 \leq i,j \leq m}$ and
$Y = (Y_{i,j})_{1 \leq i,j \leq m}$ as block-matrices,
with $X_{i,j} , Y_{i,j} \in \mathbb{M}_{d_i \times d_j} (R)$ and $X_{i,j} = 0$ for $i > j$.
We have $Y_{1,1}X_{11} = I_{d_1}$,
so by the direct finiteness hypothesis
and Lemma \ref{lem:DFsubobject},
$X_{1,1}$ is invertible.
Now, $Y_{j,1}X_{1,1}=0$ for $2 \leq j \leq m$,
so $Y_{j,1}=0$.
Thus, letting $X' , Y' \in \mathbb{M}_{D-d_1} (R)$
be the bottom-right blocks of $X$ and $Y$,
we have $X'Y'=I_{D-d_1}$.
Since $X' \in B_{\mathbf{d}'} (R)$,
where $\mathbf{d}' = (d_2,\ldots,d_m)$,
we have $Y' \in B_{\mathbf{d}'} (R)$
by inductive hypothesis, so $Y \in B_{\mathbf{d}} (R)$.
\end{proof}

\begin{lemma} \label{lem:UUTGrp}
Let $R$ be a unital ring and let $U_D (R)$ be the set of $D$-by-$D$ upper-unitriangular
matrices over $R$. Then $U_D (R)$ forms a group under matrix multiplication.
\end{lemma}

\begin{proof}
For $0 \leq k \leq (D-1)$ let $U_{D,k} (R) \subseteq U_D (R)$ be the set of matrices $A \in U_D(R)$
satisfying $A_{i,j} = 0$ for $1 \leq j-i \leq k$,
so that $U_{D,0} (R)=U_D(R)$; $U_{D,k+1} (R) \subseteq U_{D,k} (R)$
and $U_{D,D-1} (R) = \lbrace I_D \rbrace$.
We claim that every element of $U_D(R)$ has a right-inverse.
By induction on $k$, it suffices to show that, if $A \in U_{D,k}(R)$
then there exists $B \in U_{D,k}(R)$ such that $AB \in U_{D,k+1}(R)$.
Taking $B = 2 I_D -A$ suffices.
\end{proof}

\begin{lemma} \label{BlockDiagEquivLem}
Keeping the terminology from Notation \ref{BlockMxNotn},
$B_{\mathbf{d}} (R)$ is directly finite if and only if
$\mathbb{M}_{\max_i (d_i)} (R)$
is directly finite.
\end{lemma}

\begin{proof}

The forward direction is an application of Lemma \ref{lem:DFsubobject}.

Conversely, let $U_D (R)$ be the group of upper-unitriangular matrices over $R$
(by Lemma \ref{lem:UUTGrp} this is indeed a group);
let:
\begin{equation*}
X = \left( \begin{array}{cccc}
A_1 & \ast & \cdots & \ast \\
 0 & A_2 & \ddots & \vdots \\
 \vdots & \ddots & \ddots & \ast \\
 0 & \cdots & 0 & A_m
\end{array} \right)
\text{ and }
Y = \left( \begin{array}{cccc}
B_1 & \ast & \cdots & \ast \\
 0 & B_2 & \ddots & \vdots \\
 \vdots & \ddots & \ddots & \ast \\
 0 & \cdots & 0 & B_m
\end{array} \right)
\in B_{\mathbf{d}} (R)
\end{equation*}
(with $A_i , B_i \in \mathbb{M}_{d_i} (R)$),
and suppose that $XY = I_D$.
Then for $1 \leq i \leq m$, $A_i B_i = I_{d_i}$,
so by direct finiteness of $\mathbb{M}_{d_i} (R)$,
all $B_i A_i = I_{d_i}$ also, hence $YX \in U_D (R)$.
It follows that $Y$ has a right-inverse $X (YX)^{-1}$,
but since $X$ is a left-inverse to $Y$,
$X = X (YX)^{-1}$ by Lemma \ref{lem:left:right}.
Thus $YX = I_D$ also.
\end{proof}


\begin{proof}[Proof of Theorem \ref{AbBaseIffThm}]
The case in which $\Gamma$ is finite can be dealt with as in the beginning of the proof of Theorem \ref{PGrpStabFinThm}, so once again we may assume that $\Gamma$ is infinite and use the results from Section \ref{s:wreath}.

For $1 \leq i \leq m$ and $1 \leq j \leq d_i$,
let $a_{i,j} \in P$ generate the $j$th $(\mathbb{Z}/p^i\mathbb{Z})$-factor of $P$.
Let $Q \leq P$ be the set of elements of $P$ of order dividing $p$.
Then $Q$ is an elementary abelian $p$-group with basis
$\lbrace b_{i,j} : 1 \leq i \leq m , 1 \leq j \leq d_i \rbrace$,
where $b_{i,j} = p^{i-1} a_{i,j} $,
and $Q [\Gamma]$ is a free left $\mathbb{F}_p[\Gamma]$-module
of rank $D$, with free basis
$\mathcal{B} = \lbrace b_{i,j}1_{\Gamma} : 1 \leq i \leq m , 1 \leq j \leq d_i \rbrace$.
For $1 \leq i \leq m$ let $W_i \leq Q [\Gamma]$ be the left
$\mathbb{F}_p [\Gamma]$-submodule generated by
$\lbrace b_{i,j}1_{\Gamma} : 1 \leq j \leq d_i \rbrace$,
and let $V_i = W_i + W_{i+1} + \ldots + W_m \leq Q [\Gamma]$.

Suppose first that
$\mathbb{M}_{\max_i (d_i)} \big( \mathbb{F}_p [\Gamma]\big)$
is directly finite.
By Lemma \ref{BlockDiagEquivLem},
$B_{\mathbf{d}}\big( \mathbb{F}_p [\Gamma]\big)$
is directly finite.
We proceed by induction on the exponent of $P$.
The base case $\exp(P)=p$ follows from Theorem \ref{PGrpStabFinThm}.
Let $\varphi :P\wr \Gamma \to P \wr\Gamma$
be an epimorphism.
By Proposition \ref{prop:basic:morphism},
we may assume that there
is a left-$\Gamma$-equivariant epimorphism
$\psi \colon P[\Gamma] \rightarrow P[\Gamma]$
such that for all $v \in P[\Gamma]$
and $g \in \Gamma$, $\varphi(v,g)=(\psi(v),g)$ (see the proof of Theorem \ref{PGrpStabFinThm}).
The set of elements of $P[\Gamma]$ of additive order dividing $p$
is precisely $Q [\Gamma]$, so $\psi (Q[\Gamma]) \leq Q[\Gamma]$.
Thus $\varphi$ descends to a well-defined
surjective homomorphism $\overline{\varphi}$ from
$(P \wr \Gamma)/Q[\Gamma] \cong (P/Q) \wr \Gamma$
to itself.
Since $P/Q \cong (\mathbb{Z}/p\mathbb{Z})^{d_2} \oplus \cdots \oplus (\mathbb{Z}/p^{m-1}\mathbb{Z})^{d_m}$ has smaller exponent than $P$,
and by assumption $B_{\mathbf{d}_i}\big( \mathbb{F}_p [\Gamma]\big)$ is directly finite for $2 \leq i \leq m$,
$\overline{\varphi}$ is an isomorphism, by induction.
It follows that $\ker (\varphi) \leq Q [\Gamma]$.

We claim that $\psi (Q[\Gamma]) = Q[\Gamma]$.
Let $v \in Q[\Gamma]$.
Since $\varphi (P[\Gamma]) = P[\Gamma] \geq Q[\Gamma]$,
there exists $u \in P[\Gamma]$ with
$\varphi (u) = v$.
Thus $\overline{\varphi} (u + Q[\Gamma])
= v + Q[\Gamma] = 0$,
so by injectivity of $\overline{\varphi}$,
$u \in Q[\Gamma]$, as required.
Thus the restriction of $\psi$ to
$Q[\Gamma]$ is a left-$\Gamma$-equivariant
epimorphism, hence an epimorphism of left
$\mathbb{F}_p [\Gamma]$-modules.

As in the proof of Lemma \ref{lem:HopfModulesDFMatrix},
the left $\mathbb{F}_p [\Gamma]$-module homomorphism
$\psi|_{Q[\Gamma]}$ has a right-inverse
$\rho \colon Q[\Gamma] \rightarrow Q[\Gamma]$.
Let $X$ and $Y$ be the matrices of
$\psi|_{Q[\Gamma]}$ and $\rho$, respectively,
with respect to the basis $\mathcal{B}$,
so that $YX=I_D$ (note that our matrices act on the right, 
to ensure left-$\Gamma$-equivariance, so the order of matrix multiplication 
is the reverse of that for function composition). 
We claim that $X \in B_{\mathbf{d}_1}\big( \mathbb{F}_p [\Gamma]\big)$.
It will then follow from Lemma \ref{lem:BlockStrongLeftUnit}
and direct finiteness of
$B_{\mathbf{d}_1}\big( \mathbb{F}_p [\Gamma]\big)$
that $XY = I_D$ and
$\rho \circ (\psi|_{Q[\Gamma]}) = \id_{Q[\Gamma]}$,
so that $\psi|_{Q[\Gamma]}$ is injective.
Since, from the above,
$\ker (\varphi) = \ker (\psi|_{Q[\Gamma]})$,
$\varphi$ is an isomorphism, and we conclude that
$P \wr \Gamma$ is Hopfian.

It suffices then to show that $X \in B_{\mathbf{d}_1}\big( \mathbb{F}_p [\Gamma]\big)$.
In other words,
it suffices to show that for each $1 \leq i \leq m$,
$\psi(W_i) \leq V_i$.
Recall that $W_i$ is generated as an
abelian group by the elements
$b_{i,j} g = p^{i-1} a_{i,j} g$
(for $1 \leq j \leq d_i$ and $g \in \Gamma$),
which are $p^{i-1}$th powers in the group $P[\Gamma]$.
Thus $\psi (b_{i,j} g)$ is a
$p^{i-1}$th power in $P[\Gamma]$.
Since $P [\Gamma]$ is generated as an abelian
group by all elements of the form $a_{k,l} h$
(for $1 \leq k \leq m$;
$1 \leq l \leq d_k$ and $h \in \Gamma$),
and $p^{i-1} a_{k,l} h = 0$ for $k \leq i-1$,
it follows that $\psi (b_{i,j} g)$
lies in:
\begin{center}
$Q [\Gamma] \cap \big\langle p^{i-1} a_{k,l} h : k \geq i , 1 \leq l \leq d_k , h \in \Gamma \big\rangle$
\end{center}
which is precisely $V_i$, as desired.

\medskip

Conversely suppose that $P \wr \Gamma$ is Hopfian.
Then by Lemma \ref{lem:AB:Hopfian},
$(\mathbb{Z}/p^i\mathbb{Z})^{d_i} \wr \Gamma$ is Hopfian for each
$1 \leq i \leq m$.
Thus by Theorem \ref{PGrpStabFinThm},
each ring $\mathbb{M}_{d_i} \big( (\mathbb{Z}/p^i\mathbb{Z})[\Gamma] \big)$
is directly finite.
Finally, by Lemma \ref{lem:Hensel},
each ring $\mathbb{M}_{d_i} \big( \mathbb{F}_p[\Gamma] \big)$
is directly finite.
\end{proof}

We conclude our study of wreath products with abelian
bases by completing the proof of Theorem
\ref{intro:thm:main:general} from the Introduction,
the statement of which we recall here.

\begin{theorem}\label{thm:main:general}
The following are equivalent:
\begin{enumerate}
    \item For every finitely generated abelian group $A$ and every finitely generated Hopfian group $\Gamma$, the wreath product $A \wr \Gamma$ is Hopfian.
    \item Kaplansky's direct finiteness conjecture holds.
\end{enumerate}
\end{theorem}

\begin{proof}
Kaplansky's direct finiteness conjecture is equivalent to Kaplansky's stable finiteness conjecture by Theorem \ref{thm:df:equivalence}.
By Theorem \ref{intro:thm:main} and Corollary \ref{cor:kaplansky:Fp} the first item is equivalent to the fact that for every finitely generated Hopfian group $\Gamma$ and every field $\mathbb{F}$, the group ring $\mathbb{F}[\Gamma]$ is stably finite. By Lemma \ref{lem:DFsubobject}, it suffices to notice that every finitely generated group embeds into a finitely generated Hopfian group \cite{hopf:embedding}.
\end{proof}

\section{Beyond abelian bases}
\label{s:nonabelian}

We end by considering cases in which the base is not abelian. We start by looking at certain centreless bases, in which case the structure will be such that Hopficity is much easier to show, and in particular we will obtain examples of Hopfian wreath products $\Delta \wr \Gamma$ where $\Gamma$ is non-Hopfian (Theorem \ref{thm:nonhopfiantop}).
Secondly, we extend the results on abelian bases to nilpotent bases (Theorem \ref{thm:nilpotentbasis}).

\subsection{Centreless bases}

The proof of the following proposition is essentially contained in the literature \cite{grunberg, hall}. We include a proof for the reader's convenience. 

\begin{proposition}
\label{prop:dich}

Let $N \leq \Delta \wr \Gamma$ be a normal subgroup.
\begin{enumerate}
\item Suppose $N$ is \emph{not} basic. Then $N$ contains $\Delta'[\Gamma]$.
\item Suppose $N$ is basic, and let $K$ be the projection of $N$ onto $\Delta_{(e)}$. Then $K$ is a nontrivial normal subgroup of $\Delta$, 
and $([K, \Delta])[\Gamma] \leq N \leq K[\Gamma]$.
\end{enumerate}
\end{proposition}

\begin{proof}
For the first part, suppose $N$ is not basic, 
and let $(f,\gamma) \in N$ 
with $1_{\Gamma} \neq \gamma$. 
Then for any $\delta,\eta \in \Delta$ and $g \in \Gamma$, 
\begin{equation*}
\big[ (f,\gamma),\delta^{-1}_{(g)} \big] 
= \big( f(\gamma g) \delta^{-1} f(\gamma g)^{-1} \big)_{(\gamma g)} \cdot \delta_{(g)} \in N
\end{equation*}
so: 
\begin{equation*}
\big[ [(f,\gamma),\delta^{-1}_{(g)}],\eta_{(g)} \big] 
= [\delta,\eta]_{(g)} \in N
\end{equation*}
and the set of all elements of the form $[\delta,\eta]_{(g)}$ 
generates $\Delta'[\Gamma]$. 

For the second part, since for any $g \in \Gamma$ we have $N = g N g^{-1}$,
the projection of $N$ to $\Delta_{(g)}$ 
is also equal to $K$. 
Thus $N \leq K[\Gamma]$. For the other inclusion, 
given $k \in K$ and $g \in \Gamma$, 
let $f \in \Delta[\Gamma]$ with $f(g) = k$ 
(such $f$ exists by the above). 
Then for any $\delta \in \Delta$, 
$[ f,\delta_{(g)} ] = [k,\delta]_{(g)} \in N$. 
\end{proof}

Recall that a group $\Gamma$ is called 
\emph{just-non-solvable} if $\Gamma$ is not solvable, 
but every proper quotient of $\Gamma$ is solvable. 
For instance every non-abelian simple group is 
just-non-solvable. 
Thompson's group $F$ is an example of a finitely generated just-non-solvable group which is not simple, 
as $F'$ is infinite simple and contained in every 
nontrivial normal subgroup of $F$ \cite{thompson}. 
Just-infinite $p$-torsion groups, such as Grigorchuk's groups and the Gupta-Sidki $p$-groups, 
are further examples, as is the Basilica group \cite{basilica}. 

\begin{proposition} \label{prop:JNSWP}
If $\Gamma$ is solvable and $\Delta$ is 
just-non-solvable, 
then $\Delta \wr \Gamma$ is just-non-solvable. 
\end{proposition}

\begin{proof}
A central extension of a solvable group is solvable, 
so $\Delta$ is centreless. 
It follows that $\Delta '$ and $[K,\Delta]$ 
are nontrivial normal subgroups of $\Delta$ 
for every nontrivial $K \vartriangleleft \Delta$. 
By Proposition \ref{prop:dich}, 
for any nontrivial $N \vartriangleleft \Delta \wr \Gamma$ 
there exists a nontrivial $L \vartriangleleft \Delta$ 
such that $L[\Gamma] \leq N$. 
Thus $(\Delta \wr \Gamma)/N$ is a quotient of 
$(\Delta \wr \Gamma)/L[\Gamma] \cong (\Delta/L \wr \Gamma)$, 
which is solvable, since $\Delta/L$ and $\Gamma$ are. 
\end{proof}

Using Proposition \ref{prop:JNSWP}, 
we complete the proof of Theorem \ref{intro:thm:nonhopfiantop},
the statement of which we recall.

\begin{theorem}
\label{thm:nonhopfiantop}

There exist finitely generated groups 
$\Delta$ and $\Gamma$, with $\Gamma$ non-Hopfian, 
such that $\Delta \wr \Gamma$ is Hopfian. 
\end{theorem}

\begin{proof}
Let $\Delta$ be any finitely generated just-non-solvable 
group (for instance a finite nonabelian simple group). 
Let $\Gamma$ be any finitely generated 
non-Hopfian solvable group 
(for instance the Abels group \cite{Abels}). 
Then by Proposition \ref{prop:JNSWP}, 
$\Delta \wr \Gamma$ is just non-solvable. 
The result follows, 
since every just non-solvable group is Hopfian. 
\end{proof}

\begin{remark}
\label{rem:justnonP}

The previous proof admits a far-reaching generalization. 
Let $\mathcal{P}$ be a property of groups such that: 
\begin{enumerate}
\item If $\Gamma$ has $\mathcal{P}$, 
then so does every subgroup of $\Gamma$; 

\item If $\Gamma$ and $\Delta$ are groups with 
$\mathcal{P}$, then $\Delta \wr \Gamma$ 
has $\mathcal{P}$. 
\end{enumerate}

Then $\Delta \wr \Gamma$ is just-non-$\mathcal{P}$ 
whenever $\Gamma$ has $\mathcal{P}$ and 
$\Delta$ is just-non-$\mathcal{P}$ and centreless
(note that if a central extension of a group with $\mathcal{P}$ has $\mathcal{P}$, then just-non-$\mathcal{P}$ groups are automatically centreless;
this was the case for solvability). 
In particular this implies that $\Delta \wr \Gamma$ 
is Hopfian. 
Examples of properties $\mathcal{P}$ satisfying the above include ``finite''; ``amenable''; 
``torsion-free''; ``torsion''; ``left orderable''; and ``sofic'' (by \cite{wreathsofic}).
\end{remark}

\subsection{Nilpotent bases}

In this subsection, we extend the case of abelian bases to all nilpotent bases:

\begin{theorem}
\label{thm:nilpotentbasis}

Let $\Delta$ be a finitely generated nilpotent group, and let $\{ Z_i \}_{i = 0}^c$ be the upper central series of $\Delta$, so that $Z_1 = Z(\Delta)$.
Let $\Gamma$ be a finitely generated group, and suppose that $(Z_i / Z_{i-1}) \wr \Gamma$ is Hopfian for all $i = 1,\ldots,c$.
Then $\Delta \wr \Gamma$ is Hopfian.
\end{theorem}

\begin{corollary}\label{cor:nilpotentbasis}
The following are equivalent:
\begin{enumerate}
    \item For every finitely generated nilpotent group $\Delta$ and every finitely generated Hopfian group $\Gamma$, the wreath product $\Delta \wr \Gamma$ is Hopfian.
    \item Kaplansky's stable finiteness conjecture holds.
\end{enumerate}
\end{corollary}

\begin{proof}
This follows by combining Theorem \ref{intro:thm:main:general} with Theorem \ref{thm:nilpotentbasis}. Note that in order to apply Theorem \ref{intro:thm:main:general} to the terms of the upper central series, we are using the fact that every subgroup of $\Delta$ is finitely generated \cite[Section 15.3]{robinson}.
\end{proof}

The proof of Theorem \ref{thm:nilpotentbasis} will be by induction on $c$. 
The induction step is of independent interest.

\begin{lemma}
\label{lem:centre:basic}

Let $\Delta$ be a non-abelian group, and let $Z \coloneqq Z(\Delta)$ be its centre. Let $\Gamma$ be a Hopfian group. Then every epimorphism $\varphi \colon \Delta \wr \Gamma \to \Delta \wr \Gamma$ satisfies $\varphi(Z[\Gamma]) \leq Z[\Gamma]$.
\end{lemma}

\begin{proof}
We start by showing that $\varphi(\Delta[\Gamma]) \cap \Delta[\Gamma]$ has a surjective projection onto $\Delta_{(e)}$. In case $\varphi$ is basic, this follows from Lemma \ref{lem:epi:semi}. Otherwise let $(f, \gamma) \in \varphi(\Delta[\Gamma])$ be such that $\gamma \neq 1_{\Gamma}$. We then argue as in Lemma \ref{lem:abeliannormal}: Since $\varphi(\Delta[\Gamma])$ is normal (because $\varphi$ is surjective), it must also contain, for any $\delta \in \Delta$, 
$[(f, \gamma), (\delta^{-1})_{(e)}] = {_f}(\delta^{-1})_{(\gamma)} \cdot \delta_{(e)}$, 
whose projection onto $\Delta_{(e)}$ is $\delta$.

In fact, the above argument shows that if $\Lambda \leq \Delta$ is a normal subgroup such that $\varphi(\Lambda[\Gamma])$ is non-basic, then $\varphi(\Lambda[\Gamma]) \cap \Delta[\Gamma]$ has a surjective projection onto $\Delta_{(e)}$. Since $\Delta_{(e)} \cong \Delta$ is non-abelian, this implies that $\varphi(Z[\Gamma]) \leq \Delta[\Gamma]$.

Now, $\varphi(Z[\Gamma])$ and $\varphi(\Delta[\Gamma]) \cap \Delta[\Gamma]$ 
commute elementwise, and since the projection of the latter to $\Delta_{(e)}$ 
is surjective, it follows that the projection of $\varphi(Z[\Gamma])$ to $\Delta_{(e)}$ 
has image contained in $Z_{(e)}$. 
Finally, for every $\gamma \in \Gamma$, $\Delta_{(\gamma)}$ 
is conjugate to $\Delta_{(e)}$ in $\Delta \wr \Gamma$, 
so by normality of $\varphi(Z[\Gamma])$ in $\Delta \wr \Gamma$ 
we have that that the projection of $\varphi(Z[\Gamma])$ to $\Delta_{(\gamma)}$ 
has image contained in $Z_{(\gamma)}$. 
Thus $\varphi(Z[\Gamma]) \leq Z[\Gamma]$. 
\end{proof}

\begin{proposition}
\label{prop:centre:hopfian}

Let $\Delta, \Gamma$ be groups, with $\Gamma$ finitely generated and infinite, and let $Z \coloneqq Z(\Delta)$. Suppose that $Z \wr \Gamma$ is Hopfian, and that $(\Delta / Z) \wr \Gamma$ is Hopfian, and every automorphism of $(\Delta / Z) \wr \Gamma$ is basic. Then $\Delta \wr \Gamma$ is Hopfian, and every automorphism of $\Delta \wr \Gamma$ is basic.
\end{proposition}

\begin{proof}
Let $\varphi \colon \Delta \wr \Gamma \to \Delta \wr \Gamma$ be an epimorphism. If $\Delta$ is abelian, then $\Delta = Z$, and we are done by Proposition \ref{prop:morphismsarebasic}. Otherwise we can apply Lemma \ref{lem:centre:basic} and we have $\varphi(Z[\Gamma]) \leq Z[\Gamma]$ 
(noting the assumption that $Z \wr \Gamma$ is Hopfian implies $\Gamma$ is Hopfian, by Lemma \ref{lem:G:Hopfian}). We now proceed as in the proof of Lemma \ref{lem:epi:semi}:

\[\begin{tikzcd}
	{\Delta \wr \Gamma} && {\Delta \wr \Gamma} \\
	\\
	{(\Delta/Z) \wr \Gamma} && {(\Delta \wr \Gamma)/\varphi(Z[\Gamma])}
	\arrow["\varphi", two heads, from=1-1, to=1-3]
	\arrow[two heads, from=1-1, to=3-1]
	\arrow[two heads, from=1-3, to=3-3]
	\arrow[two heads, from=3-1, to=3-3]
\end{tikzcd}\]

By assumption $(\Delta / Z) \wr \Gamma$ is Hopfian. Since $(\Delta \wr \Gamma)/\varphi(Z[\Gamma])$ surjects onto $(\Delta/Z) \wr \Gamma$, it follows that $\varphi(Z[\Gamma]) = Z[\Gamma]$, and $\varphi$ descends to an automorphism of $(\Delta / Z) \wr \Gamma$, which is basic by hypothesis. It follows that $\varphi$ itself is basic, and so we may write $\varphi(\gamma) = (b(\gamma), \alpha(\gamma))$, for a map $b \colon \Gamma \to \Delta[\Gamma]$ and an automorphism $\alpha \colon \Gamma \to \Gamma$.

Define $\psi \colon Z \wr \Gamma \to Z \wr \Gamma$ by $\psi|_{Z[\Gamma]} = \varphi|_{Z[\Gamma]}$, and $\psi|_\Gamma = \alpha$. Then $\psi$ is a homomorphism: Indeed, it is clearly a homomorphism on $Z[\Gamma]$ and on $\Gamma$, and it satisfies the conjugacy relation because $\Delta[\Gamma]$ centralizes $Z[\Gamma]$. Moreover it is surjective, since $\varphi(Z[\Gamma]) = Z[\Gamma]$ and $\alpha(\Gamma) = \Gamma$. Since $Z \wr \Gamma$ is Hopfian by assumption, $\psi$ is injective. Since $\varphi$ induces an automorphism of $(\Delta / Z) \wr \Gamma = (\Delta \wr \Gamma) / Z[\Gamma]$, we have $\ker(\varphi) \leq Z[\Gamma]$ and using that $\psi|_{Z[\Gamma]} = \varphi|_{Z[\Gamma]}$, we conclude that $\varphi$ is also injective.
\end{proof}

We now easily obtain Theorem \ref{thm:nilpotentbasis}:

\begin{proof}[Proof of Theorem \ref{thm:nilpotentbasis}]

If $\Gamma$ is finite, then $\Delta \wr \Gamma$ is finitely generated and residually finite by \cite{grunberg}, therefore it is Hopfian (here we are using that finitely generated nilpotent groups are residually finite \cite[Section 15.4]{robinson}). Moreover since $Z_1$ is abelian, and $Z_1 \wr \Gamma$ is Hopfian, by Lemma \ref{lem:G:Hopfian} we obtain that $\Gamma$ is Hopfian. Thus from now on, we let $\Gamma$ be a fixed finitely generated infinite Hopfian group.

We will prove the following stronger statement by induction. Let $\Delta$ be a finitely generated nilpotent group, and let $\{ Z_i \}_{i = 0}^c$ be the upper central series of $\Delta$. 
Suppose that $(Z_i / Z_{i-1}) \wr \Gamma$ is Hopfian for all $i = 1, \ldots, c$. Then $\Delta \wr \Gamma$ is Hopfian, and moreover every automorphism of $\Delta \wr \Gamma$ is basic. The case $c = 1$ is already included in Proposition \ref{prop:centre:hopfian}.
Now suppose that $c > 1$. Then $\Delta / Z$ is a finitely generated nilpotent group of class $(c - 1)$ with upper central series $\{ Z_i \}_{i = 1}^c$. Therefore by induction $(\Delta / Z) \wr \Gamma$ is Hopfian, and every automorphism is basic. We conclude by Proposition \ref{prop:centre:hopfian} that $\Delta \wr \Gamma$ is Hopfian, and every automorphism is basic.
\end{proof}

\section{Co-Hopfian groups}

Following circulation of a preliminary version of this article, 
we have been asked several times what can be said about the \emph{co-Hopf} property 
for wreath products (recall that a group $\Gamma$ is \emph{co-Hopfian} if every 
monomorphism $\Gamma \rightarrow \Gamma$ is an isomorphism). 
Examples of co-Hopfian groups include finite groups, and several important groups from geometric group theory, including one-ended hyperbolic groups \cite{hyphopf1, hyphopf3}, mapping class groups \cite{cohopf:mcg} and outer automorphism groups of free groups \cite{cohopf:out}.
Non-examples include free groups, or more generally right-angled Artin groups, and infinite finitely generated abelian groups.

\medskip

In one respect, the co-Hopf property for wreath products is better behaved than the Hopf property, in that examples as in Theorem \ref{intro:thm:nonhopfiantop} do not occur:

\begin{proposition}
Let $\Gamma$ and $\Delta$ be groups and suppose $\Delta \wr \Gamma$ is co-Hopfian. 
Then $\Gamma$ and $\Delta$ are co-Hopfian. 
\end{proposition}

\begin{proof}
Let $\phi \colon \Delta \rightarrow \Delta$ and $\psi \colon \Gamma \rightarrow \Gamma$ 
be monomorphisms. Then we have monomorphisms 
$\Phi , \Psi \colon \Delta \wr \Gamma \rightarrow \Delta \wr \Gamma$ given by: 
\[\Phi(f,\gamma) = (\phi \circ f,\gamma) \text{ and } \Psi(f,\gamma) = (\overline{f},\psi(\gamma))\]
where $\overline{f} : \Gamma \rightarrow \Delta$ is given by: 
\[\overline{f} (\gamma) = 
\begin{cases}
f(\gamma') \text{ if } \gamma = \psi(\gamma'); \\
1_{\Delta} \text{ if } \gamma \notin \im(\psi).
\end{cases}
\]
If $\phi$ (respectively $\psi$) is not surjective, then neither is $\Phi$ (respectively $\Psi$). 
\end{proof}

It seems not unreasonable to expect that, if $\Gamma$ is a co-Hopfian group and $A$ is 
an abelian finite (hence co-Hopfian) group, 
then the existence (or not) of \emph{basic} monomorphisms 
$A \wr \Gamma \rightarrow A \wr \Gamma$ which are not surjective 
should be related to stable finiteness properties of $\Gamma$. 
Unfortunately, there does not presently seem to be much hope of 
reducing \emph{all} monomorphisms to 
basic monomorphisms, since the image of the base group under a monomorphism 
need not be normal in $A \wr \Gamma$, 
hence Lemma \ref{lem:abeliannormal} is no longer directly relevant.
We leave the analogue of Question \ref{q:main} as a further open question:

\begin{question}
Let $\Delta, \Gamma$ be finitely generated groups. When is the wreath product $\Delta \wr \Gamma$ co-Hopfian?
\end{question}

\footnotesize

\bibliographystyle{amsalpha}
\bibliography{ref}

\normalsize

\vspace{0.5cm}

\noindent{\textsc{Christ's College, Cambridge, UK}}

\noindent{\textit{E-mail address:} \texttt{hb470@cam.ac.uk}} \\

\noindent{\textsc{Department of Pure Mathematics and Mathematical Statistics, University of Cambridge, UK}}

\noindent{\textit{E-mail address:} \texttt{ff373@cam.ac.uk}}

\end{document}